\documentclass{article}

\usepackage{amsmath, amssymb, latexsym}

\usepackage{amsthm}
\usepackage{graphs}

\newtheorem{theorem}{Theorem}[section]
\newtheorem{proposition}[theorem]{Proposition}
\newtheorem{lemma}[theorem]{Lemma}
\newtheorem{corollary}[theorem]{Corollary}

\theoremstyle{definition}
\newtheorem{definition}[theorem]{Definition}
\theoremstyle{remark}
\newtheorem{remark}[theorem]{Remark}
\theoremstyle{remark}

\theoremstyle{remark}

\def\le{\leqslant}

\begin{document}

\title{On Characters of Inductive  Limits of  Symmetric Groups%\thanks{Grants or other notes
}


\author{ {\bf Artem Dudko}  \\
                    University of Toronto,  Toronto,  Canada  \\
          artem.dudko@utoronto.ca \\
         {\bf Konstantin Medynets} \\
         Ohio State University, Columbus, Ohio, USA \\
         medynets@math.ohio-state.edu }

\date{}

\maketitle

\begin{abstract}
In the paper we completely  describe characters  (central positive-definite functions) of simple locally finite groups that can be represented as inductive limits of (products of) symmetric groups under block diagonal embeddings. Each such  group $G$ defines an infinite graph (Bratteli diagram) that encodes the  embedding scheme. The group $G$ acts on the space $X$  of infinite paths of the associated Bratteli diagram by changing initial edges of paths.  Assuming the  finiteness of the set of ergodic measures for the system $(X,G)$, we establish that each indecomposable character $\chi :G \rightarrow \mathbb C$ is  uniquely defined by the formula $\chi(g) = \mu_1(Fix(g))^{\alpha_1}\cdots \mu_k(Fix(g))^{\alpha_k}$, where $\mu_1,\ldots,\mu_k$ are $G$-ergodic measures, $Fix(g) = \{x\in X: gx = x\}$, and $\alpha_1,\ldots,\alpha_k\in  \{0,1,\ldots,\infty\}$.  %As a corollary, we  get that any factor representation of the group $G$ built by a non-regular indecomposable character is automatically continuous with respect to the uniform topology (introduced by P.~Halmos in ergodic theory).
 We illustrate our results  on the group of rational permutations of the unit interval.
\end{abstract}

\noindent{\small MSC: 20C32,  20B27,  37B05. }

\section{Introduction} The representation theory of groups, originally motivated by the study of finite groups, eventually turned into a powerful framework for formulating and solving problems from numerous areas of mathematics and theoretical physics.  A later development of the representation theory of infinite groups connected the group theory to such  mathematical areas as  the free probability theory, random matrices, algebraic geometry, and  many others.  This was mainly   established via the group  $S(\infty)$ of all finitary permutations of the set of natural numbers, see the discussion and references in  the monograph \cite{Kerov:book}.

The classification of characters for the group $S(\infty)$ was obtained   by Thoma \cite{thoma:1964} in 1964. Later  Kerov and Vershik \cite{vershik_kerov:1981_1, vershik_kerov:1981} showed that  indecomposable characters of $S(\infty)$ can be found as  weak limits of characters of groups $S(n)$.  They called this construction the {\it ergodic approach or the asymptotic theory of characters}. We note that the application of the ergodic approach in \cite{vershik_kerov:1981} provided a very neat interpretation of Thoma's parameters as  frequencies of rows and columns in the  Young diagram.
 %Furthermore, the asymptotic approach   can be  applied to other groups, see, for example, \cite{goryachko_petrov:2010},  \cite{vershik:2001}, and references therein.
We also mention     a recent preprint \cite{gohm_kostler:2010}, where the authors give a purely operator algebra proof of the classification of characters for $S(\infty)$, and the paper \cite{Okounkov:1997} by Okounkov, who used  the Olshanski semigroup approach to give another proof  of Thoma's theorem.

In the present paper we are interested in the theory of characters for groups that can be represented as inductive limits of  symmetric groups under diagonal embeddings. One of the simplest examples of such a group can be given as follows.
 Consider the sequence of sets $X_n = \{0,\ldots,2^n-1\}$ and the symmetric groups $S(X_n)$.  The set $X_{n+1}$ can be represented as a disjoint union of two copies of $X_n$. Then each element $s\in S(X_n)$  can be embedded into $S(X_{n+1})$ by making it act on each copy of $X_n$ as $s$.  Define the group $S(2^\infty)$  as the inductive limit  of groups $S(X_n)$ under this embedding scheme.  The characters of $S(2^\infty)$ were completely described by the first-named author in \cite{dudko:2011}.     We would like to mention that the technique used  in \cite{dudko:2011}, which is further developed in the present paper, is significantly different from the ergodic approach of Vershik and Kerov and is solely based on the study of  actions of $S(2^\infty)$ and {\it not} on the analysis of weak limits of characters.

 Characters of  groups similar to $S(2^\infty)$ were also studied by Leinen and Puglisi in \cite{LeinenPuglisi:2004}, who applied a version of the ergodic approach to describe characters for inductive limits of alternating groups $A(X_n)$  with the embedding scheme allowing the set $X_{n+1}$ to be covered by multiple  copies of $X_n$.

The group   $S(2^\infty)$ can be represented as the group of homeomorphisms of $\{0,1\}^\mathbb N$ preserving the tails of sequences or as the  group permuting initial segments of infinite paths of  the Bratteli diagram corresponding to the 2-odometer. Many groups covered in \cite{LeinenPuglisi:2004} can be also seen as groups of permutations of Bratteli diagrams (corresponding to odometers).

Following this observation,   we will consider groups acting on  initial segments of infinite paths of {\it general} Bratteli diagram,  see Definition \ref{DefinitionFullGroup}. We will refer to these groups as {\it full groups}\footnote{These groups are also called AF (approximately finite) full groups, see \cite{matui:2006}.} of Bratteli diagrams.
 Algebraically, full groups are  inductive limits of products of symmetric groups under  block diagonal embeddings (similar to those used in $S(2^\infty)$), see Section \ref{SubsectionInductiveLimits} for the details.    We notice that full groups are countable and locally finite. The algebraic structure of full groups and groups isomorphic to them were studied in  \cite{kroshko_sushchansky:1998, lavrenyuk_sushchanskii:2009,LavrenyukNekrashevych:2007,matui:2006,bezuglyi_medynets:2008,zalesski:1995}. A  discussion of the connection between full groups and the classification of general locally finite groups can be found in \cite{lavrenyuk_sushchanskii:2009,LavrenyukNekrashevych:2007}.

 Consider the full group $G$ of a Bratteli diagram $B$. Then the group $G$ acts on the space $X$ of infinite paths of $B$. The set $X$ can be endowed with a topology that turns it into a Cantor set. We will assume that the group $G$ is simple, which implies that the dynamical system $(X,G)$ is minimal (every $G$-orbit is dense). Fix an indecomposable character  $\chi:G\rightarrow \mathbb C$, i.e. a central positive-definite function that cannot be represented as a convex combination of similar functions. The main goal of the paper is to find a closed-form expression for $\chi$. Our approach to this problem is based on the study of action of the group $G$ on the path-space $X$. In our proofs, we  use some ideas  of \cite{dudko:2011}.

Applying the GNS-construction to the character $\chi$, one can find a   finite-type factor representation\footnote{The term ``finite-type'' means that the von Neumann algebra generated by the representation of the group $G$ is of finite type, i.e. either of type $II_1$ or $I_n$ with $n<\infty$, see the details of this classifications in \cite[Chapter 5]{Tak}.}  $\pi$ of the group $G$ in a Hilbert space $H$ and a vector $\xi\in H$ such that $\chi(g)  = (\pi(g)\xi,\xi)$, see the details in Section \ref{SubsectionCharacters}. Denote by $\mathcal M_\pi$ the $W^*$-algebra generated by the unitary operators $\pi(G)$.  Observe that $\mathcal M_\pi$ has a unique normalized trace $tr$ given by $tr(T) = (T\xi,\xi)$.

 The first main result of our paper is the proof of the following theorem establishing that the value of the character $\chi$ at the element $g\in G$ depends only on the set $supp(g) = \{x\in X : g(x)\neq x\}$. For a clopen set $A$, denote by $G(A)$ the subgroup of all elements $g\in G$ supported by the set $A$.  Let $P^A$ stand for the projection on the subspace $\{h\in H : \pi(g)h = h\mbox{ for all }g\in G(A)\}$.

\begin{theorem}\label{TheoremTracePreliminaries} Let $G$ be the simple full group of a Bratteli diagram. Then for every $g\in G$, we have that $P^{supp(g)}\in\mathcal M_\pi$ and $$\chi(g) = tr(P^{supp(g)}).$$
\end{theorem}

Observe that  groups $G(A)$   play a similar role in $G$ as the finite symmetric  groups $S(n)$ do in $S(\infty)$.

 Assume  that  the dynamical system $(X,G)$ has only a finitely many of ergodic measures $\mu_1,\ldots,\mu_k$. Under this assumption,  we will establish the following result.

 \begin{theorem}\label{TheoremFormulaPreliminaries} For each indecomposable character $\chi$  of the simple full group $G$ there exist parameters $\alpha_1,\ldots,\alpha_k\in \{0,1,\ldots,\infty\}$  such that $$\chi(g) = \mu_1(Fix(g))^{\alpha_1}\cdots \mu_k(Fix(g))^{\alpha_k}$$ for every $g\in G$. Here $Fix(g) = \{x\in X : g(x) = x\}$.
  \end{theorem}

If $\alpha_i = \infty$ for some $i$, then the character $\chi$ is regular, i.e. $\chi(1) = 1$ and $\chi(g) = 0$ for all $g\neq 1$.   Theorem \ref{TheoremFormulaPreliminaries} shows that the description of characters for each given full group   is reduced to the description of ergodic measures on the corresponding Bratteli diagram. We  mention the following papers \cite{bezuglyi_kwiatkowski_medynets_solomyak:2010,bezuglyi_kwiatkowski_medynets_solomyak:2011, ferenczi_fisher_talet:2009,Mela:2006,petersen_varchenko:2010} devoted to the study of  ergodic measures on Bratteli diagrams.

    We  observe that the characters  of  the group $S(\infty)$ that
 take only non-negative values can be also
obtained as a measure of fixed points from certain
actions of $S(\infty)$, see \cite{vershik_kerov:1981_1}.  The description of characters in terms of measures of fixed points also appears in \cite{dudko:2011}, \cite{goryachko_petrov:2010}, and  \cite{LeinenPuglisi:2004}\footnote{Formally, the characters in \cite[Theorem 3.2]{LeinenPuglisi:2004} are given as functions that depend on the cardinality of the set of fixed points. However, these functions can be interpreted as  powers of the measure of fixed points.}, see also  \cite[Chapter 9]{grigorchuk:2011} for characters of certain finitely generated groups.

Theorem \ref{TheoremFormulaPreliminaries} classifies characters
for simple full groups and, as a result, all $II_1$-factor
representations of $G$ up to quasi-equivalence.
Observe that  there are Bratteli diagrams with
 non-simple full groups. Such groups might have characters
 taking complex values, which   cannot
  be exclusively determined by invariant measures.
However, we think that characters of arbitrary full groups can be described as a product of characters ``built by measures'' and some homomorphisms into the unit circle,
 cf. \cite{vershik:2011}.
 Denote by $D(G)$ the commutator subgroup of $G$.
Notice that $D(G)$ is a simple group if the dynamical system $(X,G)$
is minimal, see \cite{matui:2006} or \cite{lavrenyuk_sushchanskii:2009}.

\bigskip\noindent{\bf Conjecture.} {\it Let $G$ be the full group (not necessarily simple) of a Bratteli diagram.  Then
 every indecomposable character $\chi$ of the group $G$ has the form
$$\chi(g) = \rho([g])\cdot \mu_1(Fix(g))
^{\alpha_1}\cdots \mu_k(Fix(g))^{\alpha_k}
\mbox{ for }g\in G,$$ where $\alpha_1,\ldots,\alpha_k\in \{0,1,\ldots,\infty\}$,
  $\mu_1,\ldots,\mu_k$ are some ergodic $G$-invariant measures,
$[g]$ is the image of $g$ in the (Abelian) quotient group $H = G/D(G)$,
and $\rho$ is a homomorphism from the  group $H$ into the unit
circle $\{z\in\mathbb{C}:|z|=1\}$. }

\bigskip The structure of the paper is the following. In Section \ref{SectionPreliminaries}, we fix the notations and introduce the main definitions of the paper. Subsection \ref{SubsectionBratteliDiagrams} is devoted to Bratteli diagrams. Here we give a formal definition of Bratteli diagrams and their full groups. We prove that the full group is simple if and only if the Bratteli diagram can be telescoped to have an even number of edges between any two vertices. In Subsection  \ref{SubsectionCharacters} we give the definition of a character and recall the essence of the GNS construction.

Section \ref{SectionAlgebraicProperties} is mostly technical and contains two lemmas that allow us to move elements of the unitary group $\pi(G)$ to the center of the algebra $\mathcal M_\pi$ by conjugation. This technique was also extensively used in  the representation theory of $S(\infty)$, see, for example, \cite{Okounkov:1997}.  We also introduce a topology (Definition \ref{DefinitionUniformTopology}) on the group $G$ similar to the {\it uniform topology} studied in ergodic theory \cite{halmos:book}.  We will use this topology to study continuous characters and representations of non-simple full groups.

In Section \ref{OrthogonalProjections}, we study orthogonal
projections $P^A$ on the subspace of  vectors fixed under the group $\pi(G(A))$.
We show that when $A$ is a clopen set,  the projection $P^A$ can be obtained
 as a weak limit of involutions $\{\pi(s_n)\}$ with $supp(s_n) = A$.
Furthermore, we show that the projections $\{P^A\}$ form an Abelian
semigroup, i.e. $P^A P^B = P^{A\cup B}$ for any clopen sets $A$ and $B$.
 We then prove Theorem \ref{TheoremTracePreliminaries}   for simple full groups. We also establish this result  for arbitrary  full groups under the assumption that the character or the representation is continuous with respect to the uniform topology.

In Section \ref{SectionFormulaCharacter} we consider full groups that have only a finitely many of ergodic measures $\mu_1,\ldots,\mu_k$. We introduce the function  $\varphi(\mu_1(A),\ldots,\mu_k(A)) = tr(P^{X\setminus A})$, $A\subset X$, defined on a dense subset of $[0,1]^k$. We show that the domain of $\varphi$ can be extended to the entire set $[0,1]^k$ and that the function is multiplicative and continuous. Note that the multiplicativity of $\varphi$ can be seen as  a version of multiplicativity of characters on $S(\infty)$, see, for example, \cite{Okounkov:1997}. Solving Cauchy's functional equation for the function $\varphi$, we get that $\varphi(t_1,\ldots,t_k) = t_1^{\alpha_1}\cdots t_k^{\alpha_k}$ for some non-negative real numbers $\alpha_1,\ldots,\alpha_k$.  We then show that  these parameters are, in fact, integers, which yields the proof of Theorem \ref{TheoremFormulaPreliminaries}. This description is also valid for arbitrary full groups provided that the character is continuous.

As a corollary of Theorem \ref{TheoremFormulaPreliminaries} we get that any finite-type factor representation of a simple full group (corresponding to a non-regular character) is automatically continuous with respect to the uniform topology. This indicates that {\it simple} full groups of Bratteli diagrams have many features of ``big'' groups such as groups of measure-preserving automorphisms or homeomorphisms of a Cantor set, which possess a stronger automatic continuity property \cite{kechris_rosendal:2007, kittrell_tsankov:2010, rosendal:2009,   rosendal_solecki:2007}.

 Theorem \ref{TheoremFormulaPreliminaries} can be also used to describe characters of  full groups  of Cantor minimal systems,  see Corollary \ref{CorollaryCharactersMinimalSystems}. This shows that the  property of unique ergodicity of minimal systems can be completely reformulated on the language of indecomposable characters.

 In Section \ref{SectionExampleRepresentations}, we build a $II_1$-factor representation for each  character from Theorem \ref{TheoremFormulaPreliminaries}, whereby we describe all  $II_1$-factor representations of   simple full groups up to quasi-equivalence.

 In Section \ref{SectionIntervalPermutations} we consider the group of all rational permutations of the unit interval. We apply Theorem \ref{TheoremFormulaPreliminaries} to get a new proof of the classification of characters for this group, which was originally established  by Goryachko and Petrov \cite{goryachko:2008,goryachko_petrov:2010}.  We mention that the  proof  in \cite{goryachko_petrov:2010} is based on the asymptotic theory of characters.

%%%%%%%
%
% PRELIMINARIES

\section{Preliminaries} \label{SectionPreliminaries}

In the section we give the definitions of  Bratteli diagrams, full groups, characters, and related dynamical notions.

\subsection{Bratteli diagrams and the associated full groups}\label{SubsectionBratteliDiagrams}  Bratteli diagrams have been studied in many papers on Cantor dynamics and operator algebras.  Since in the present paper we are mainly interested in the dynamical nature of Bratteli diagrams, our main references to the theory of such diagrams will be
\cite{herman_putnam_skau:1992} and \cite{giordano_putnam_skau:1995}.

\begin{definition}\label{Definition_Bratteli_Diagram} A {\it Bratteli diagram} is an
infinite graph $B=(V,E)$ such that the vertex set
$V=\bigcup_{i\geq 0}V_i$ and the edge set $E=\bigcup_{i\geq 1}E_i$
are partitioned into disjoint subsets $V_i$ and $E_i$ such that

(i) $V_0=\{v_0\}$ is a single point;

(ii) $V_i$ and $E_i$ are finite sets;

(iii) there exist a range map $r$ and a source map $s$ from $E$ to
$V$ such that $r(E_i)= V_i$, $s(E_i)= V_{i-1}$, and $s^{-1}(v)\neq\emptyset$, $r^{-1}(v')\neq\emptyset$ for all $v\in V$ and $v'\in V\setminus V_0$.
\end{definition}

A graphical example of a Bratteli diagram is given in Section \ref{SectionIntervalPermutations}.  The pair  $(V_i,E_i)$ is called the $i$-th level of the diagram $B$.
A finite or infinite sequence of edges $(e_i : e_i\in E_i)$ such
that $r(e_{i})=s(e_{i+1})$ is called a {\it finite} or {\it infinite path},
respectively.  For a Bratteli diagram
$B$, we denote by $X_B$ the set of infinite paths starting at the root vertex $v_0$. We will  always omit the index $B$ in the notation $X_B$ as the diagram will be always clear from the context. We endow the set $X$ with the topology generated by  cylinder sets $U(e_1,\ldots,e_n)=\{x\in X : x_i=e_i,\;i=1,\ldots,n\}$, where
$(e_1,\ldots,e_n)$ is a finite path from $B$. Then $X$ is a 0-dimensional compact metric space with
respect to this topology. Observe that the space $X$ might have isolated points, but we will assume that the space is perfect, i.e. $X$ is a Cantor set.

 Denote by $h_v^{(n)}$, $v\in V_n$, the number of finite paths connecting the root vertex $v_0$ and the vertex $v$. Given a vertex $v\in V_n$, enumerate the cylinder sets $U(e_1,\ldots,e_n)$ with $r(e_n)=v$ by $\{C_{v,0}^{(n)},\ldots,C_{v,h_v^{(n)-1}}^{(n)}\}$. Then

 \begin{equation} \Xi_n = \{C_{v,i}^{(n)} : v\in V_n,\;i=0,\ldots,h_v^{(n)}-1\}
 \end{equation}
 is a clopen partition of $X$.

Notice that  if $A$ is a clopen subset of $X_B$, then we can  always find a level $n\geq 1$ such that the set $A$ is a disjoint union of cylinder sets $U(e_1,\ldots,e_n)$ of depth $n$.  We will say then that the set $A$ is {\it compatible with the partition $\Xi_n$.}

 For multiple examples of diagrams as well as their diagrammatic representations we refer the reader to the papers \cite{bezuglyi_kwiatkowski_medynets_solomyak:2010}, \cite{bezuglyi_kwiatkowski_medynets_solomyak:2011}, \cite{herman_putnam_skau:1992}, and \cite{giordano_putnam_skau:1995}.

Fix a level $n\geq 0$. For each pair of vertices $(w,v)\in V_{n+1}\times V_{n}$, denote by $f^{(n)}_{w,v}$ the number of edges connecting them.

\begin{definition} The $|V_{n+1}|\times |V_n|$-matrix $F_n = (f_{w,v}^{(n)})$ ($v\in V_{n},\;w\in V_{n+1}$) is  called an {\it incidence matrix} of the diagram.
\end{definition}

\begin{remark} Observe that $F_0$ is a vector and $h_v^{(n)}$ coincides with the $v$-th coordinate of the vector $F_{n-1}\cdots F_0$.
\end{remark}

\begin{definition}  A Bratteli diagram $B=(V,E)$  is called {\it simple} if for any level $V_n$ there is a level $V_m$, $m>n$, such that every pair of vertices $(v,w)\in V_n\times V_m$ is connected by a finite path.
\end{definition}

\begin{definition} We say that two infinite paths $x = (x_n)_{n\geq 1}$ and $y = (y_n)_{n\geq 1}$ of a Bratteli diagram $B$ are {\it cofinal} or {\it tail-equivalent} if $x_n = y_n$ for all $n$ large enough.  The equivalence relation generated by cofinal paths is called the {\it cofinal} or {\it tail} equivalence relation.
\end{definition}

\begin{definition} Let $X$ be a compact space and $H$ a group acting by homeomorphisms of $X$. We say that the $H$-action is {\it minimal} if the $H$-orbit of every point $x\in X$, $Orb_H(x) = \{h(x) : h\in H\}$,  is dense in $X$.
\end{definition}

Fix a Bratteli diagram $B = (V,E)$ and a level $n\geq 1$.
Denote by $G_n$ the group of all homeomorphisms of $X: = X_B$
that changes only the first $n$-edges of each infinite path.  Namely, a homeomorphism
$g$ belongs to $G_n$ if and only if for any finite path
$(e_1,\ldots,e_n)$ there exists a finite path $(e_1',\ldots,e_n')$
with $r(e_n')=r(e_n)$ such that
$$ g(e_1,\ldots,e_n,e_{n+1},e_{n+2},\ldots)=(e_1',\ldots,e_n',e_{n+1},e_{n+2},\ldots)$$
for each infinite extension $(e_1,\ldots,e_n,e_{n+1},e_{n+2},\ldots)$ of the path $(e_1,\ldots,e_n)$.
 Observe that $G_n\subset G_{n+1}$.
Set $G = \bigcup_{n\geq 1}G_n$.

\begin{remark} (1) $G$ is a locally finite countable group.

(2) Given $x\in X$, the $G$-orbit of $x$ consists of all infinite paths cofinal to $x$.

(3) The simplicity of the diagram is equivalent to the minimality of the $G$-action on $X$.

(4) Given a level $n\geq 1$ and a vertex $v\in V_n$, denote by $G_{n}^{(v)}$ the subgroup of $G_n$ such that $g\in G_{n}^{(v)}$ if $g(x) = x$ for all $x\in X $ with $r(x_n)\neq v$. In other words, the elements of $G_{n}^{(v)}$ permutes only the infinite paths going through the vertex $v$. Then $$G_n = \prod_{v\in V_{n}}G_{n}^{(v)}.$$
\end{remark}

\begin{definition}\label{DefinitionFullGroup} We will call  the group $G$  the {\it full group of the Bratteli diagram $B$}. The symbol $G$ will be reserved exclusively for the full groups of Bratteli diagrams.
\end{definition}

For various (other) notions of full groups in the context of Cantor dynamics, we refer the reader to the papers \cite{giordano_putnam_skau:1999} and \cite{bezuglyi_medynets:2008}.  The main property of the group $G$, which also holds in other full groups (cf. \cite{giordano_putnam_skau:1999}), is the ability to ``glue'' elements of $G$. Namely, if $\{A_1,\ldots, A_n\}$ is a clopen partition of $X$ and $\{g_1,\ldots, g_n\}$ are elements of $G$ such that $\{g_1(A_1),\ldots, g_n(A_n)\}$ is also a clopen partition of $X$, then the homeomorphism $g$ of $X$ defined by $g(x) = g_i(x)$ (for $x\in A_i$, $i=1,\ldots,n$) {\it also belongs to} $G$.

\begin{definition} We say that two clopen sets $A$ and $B$ are {\it $G$-equivalent} if there is $g\in G$ with $g(A)  = B$.
\end{definition}

\begin{remark} (1) Since the group $G$ is amenable (locally finite), the dynamical system $(X,G)$ has at least one $G$-invariant probability measure.   Hence if $A$ and $B$ are $G$-equivalent, then $\mu(A) = \mu(B)$ for all $G$-invariant probability measures. Observe that,  in general, the converse statement is not true.

(2) Let $A$ be a clopen set such that $A$ can be written as $A =A_1\sqcup A_2$ with $A_1$ and $A_2$ being $G$-equivalent.  Choose $n\geq 1$ so that all sets  $A_1$, $A_2$, and $A$ are compatible with the partition $\Xi_n$ and $g\in G_n$. Then for every vertex $v\in V_n$, there is an even number of finite paths from the root vertex $v_0$ to $v$ lying completely in the set $A$.

\end{remark}

The following result gives  a dynamical criterion
for a group $G$ to be simple. We note that a special
case of this result, when the diagram $B$ has exactly
one vertex at every level, can be derived from Theorem 1 of \cite{kroshko_sushchansky:1998}, cf. \cite[Corollary 4.10]{matui:2006}.  Denote by $D(G)$ the commutator subgroup of $G$, i.e.
a subgroup generated by all elements of the form
$[g_1,g_2]:=g_1g_2g_1^{-1}g_2^{-1}$, $g_1,g_2\in G$.
 Clearly, $D(G)$ is a normal subgroup of $G$.

\begin{proposition}\label{PropositionSimpleDiagrams}

(1) The group $D(G)$ is  simple if and only if the associated Bratteli diagram is simple. This is also equivalent to the minimality of the dynamical system $(X,G)$.

(2) $G$ is simple, i.e. $D(G) = G$,  if and only if the diagram is simple and every clopen set $A\subset X$ can be represented as $A = A_1\sqcup A_2$ with $A_1$ and $A_2$ being $G$-equivalent.
\end{proposition}
\begin{proof} (1) The first statement was proved in \cite[Lemma 3.4]{matui:2006}, see also \cite[Theorem 2]{lavrenyuk_sushchanskii:2009}.

(2-i) First of all assume that every clopen set can be decomposed as a union of $G$-equivalent clopen sets. Fix $g\in G$. Denote by $B_k$ the set of points having $g$-period exactly $k$, i.e. $$B_k = \{x\in X : g^k(x) = x\mbox{ and }g^l(x)\neq x\mbox{ for all }1\leq l\leq k-1\}.$$ Denote by $K$ the set of all integers $k$ with $B_k\neq \emptyset$. Note that $K$ is a finite set.
Choose a clopen set $B_k^{(0)}$ with $g^{l}(B_k^{(0)})\cap B_k^{(0)} = \emptyset$ for all $l=1,\ldots,k-1$, and $B_k^{(0)}\cup g(B_k^{(0)})\cup\ldots\cup g^{k-1}(B_k^{(0)}) = B_k$.

Find $G$-equivalent clopen sets $A_k^{(1)}$ and $A_k^{(2)}$ such that $B_k^{(0)} = A_k^{(1)}\sqcup A_k^{(2)}$. Choose $h_k\in G$ such that $h_k|(X \setminus B_k^{(0)}) = id$, $h_k(A_k^{(1)}) = A_k^{(2)}$, and $h_k^2 = id$.  Define a homeomorphism $h\in G$ as follows
$$h(x) = \left\{\begin{array}{ll}g^lh_k g^{-l}(x) & \mbox{for }x\in g^{l}(B_k^{(0)}),\;0\leq l\leq k-1,\;k\in K;\\
x & \mbox{elsewhere.}  \end{array}\right.$$

Set $$C_i = \bigcup_{k\in K}\bigcup_{l=0}^{k-1}g^l(A_k^{(i)})\mbox{ for }i=1,2.$$ Define $g_i\in G$ as $g$ on $C_i$ and as $id$ elsewhere. Now, it is a routine to check that $hg_1h^{-1} = g_2$ and $g = g_1g_2 = g_1hg_1h^{-1} $. Noting that $g_1$ is isomorphic to $g_1^{-1}$, we conclude that $g\in D(G)$. Therefore, $D(G) = G$.

(2-ii) Now assume that $G$ is a simple group, i.e. $G = D(G)$. Fix a clopen set $A$.
Without loss of generality we may assume that $A=U(e_1,\ldots, e_n)$
is a cylinder set and there exists a path
$(e_1',\ldots,e_n')\neq (e_1,\ldots, e_n)$
 between $v_0$ and $v=r(e_n)$. To prove the proposition it is enough to show that there exists a level
$m>n$ such that for any vertex $w\in V_m$ the cardinality of the set $C_{v,w}$
of finite paths between $v$ and $w$ is even. Then dividing $C_{v,w}$ into subsets
$C_{v,w}^{(1)}$ and $C_{v,w}^{(2)}$ of equal cardinalities and setting
 $A_i$, $i=1,2$, to be the union of cylinders $U(e_1,\ldots,e_m)$ with
 $(e_{n+1},\ldots,e_m)\in C_{v,w}^{(i)}$, we obtain the desired partition.

 Find $m>n$ so that every element $g\in G_n$
 can be written as a product of commutators $[f,h]$ with $f,h\in G_m$.
In other words,  the group $G_n$ can be seen as a subgroup
of the alternating group $D(G_m)$.  Assume that for some $w\in V_m$ the cardinality
of $C_{v,w}$ is odd. Let $g\in G_n^{(v)}$
be the transposition $(e_1,\ldots,e_n)\leftrightarrow (e_1',\ldots,e_n')$.
 Then, the embedding of
$g$ into $G_m^{(w)}$ is a product of an odd number of transpositions.
Hence, $g$ is an odd permutation in the group $G_m^{(w)}$, i.e. $g\notin D(G_m)$,
which is a contradiction.
\end{proof}

\begin{definition} Let $B = (V,E)$ be a Bratteli diagram. Fix a sequence of integers $0 = m_0 < m_1<\ldots$. Consider a diagram $B' = (V',E')$, where $V'_n = V_{m_n}$ and $E'_n$ consists of all finite paths of $B$ between levels $V_{m_{n-1}}$ and $V_{m_n}$. Then the diagram $B'$ is called a {\it telescope} of $B$.
\end{definition}

Observe that the telescoping procedure does not change the space of infinite paths and the full group of the diagram, see \cite[Definition 2.2]{herman_putnam_skau:1992} for more details.

Using the minimality of the system $(X,G)$ and Proposition \ref{PropositionSimpleDiagrams}, one can easily establish the following result.

\begin{proposition}\label{PropositionSimplicityEvenDiagrams} Let $G$ be the full group of a Bratteli diagram $B$. If $G$ is a simple (non-trivial) group, then we can telescope a diagram $B$ so that each pair of vertices from consecutive levels of the new diagram is always connected by an {\it even} (at least two) number of edges.
\end{proposition}

\begin{definition} Diagrams with the property as in the proposition above will be called {\it even diagrams}.
\end{definition}

%%%%%%
%  SUBSECTION Inductive LIMITS

\subsection{Inductive limits of  symmetric groups} \label{SubsectionInductiveLimits}
 Fix a Bratteli diagram $B = (V,E)$ and consider the associated full group $G$. Denote by $S(M)$ the group  of all permutations of a finite set $M$.

 Denote by $\{F_n\}_{n\geq 1}$ the sequence of incidence matrices of $B$. For every level $n$ and a vertex $v\in V_n$, consider the set $M_n^{(v)}$ of all finite paths of $B$ connecting the root vertex $v_0$ and the vertex $v$. Clearly, the group $G_n^{(v)}$ can be identified with $S(M_n^{(v)})$.  Then the  structure of the diagram $B$ gives a ``natural'' decomposition
\begin{equation}\label{EquationEmbeddingGroups}M_n^{(v)} = \bigsqcup_{w\in V_{n-1}}\bigsqcup_{i=1}^{f_{v,w}^{(n)}} M_{n-1}^{(w)},\end{equation} where each set $M_{n-1}^{(w)}$ is taken with the multiplicity $f_{v,w}^{(n)}$. This decomposition, in its turn, gives a sequence of injective homomorphisms $\rho_n: G_{n-1}\rightarrow G_n$, $n\geq 1$. Thus, the group $G$ can be seen as the inductive limit of the system $(G_{n-1}\stackrel{\rho_n}{\rightarrow}G_n)_{n\geq 1}$. Conversely, using arguments similar to Theorem 5.1 \cite{LavrenyukNekrashevych:2007} one can show that any inductive limit of symmetric groups with the embedding as above is isomorphic to the full group of some Bratteli diagram.

{\it To avoid trivialities,  we will always assume that Bratteli diagrams and the associated groups are not trivial and the path spaces are Cantor sets.
Furthermore, we will  assume that the system $(X,G)$ is minimal.}

%%%%%%%%%%%%%%%%%%%%%%%%% Characters

\subsection{Characters and factor representations} \label{SubsectionCharacters}
In this subsection we recall basic definitions from the representation theory and the theory of operator algebras. For  details we refer the reader to the books
\cite{bratelli_robinson:1982}, \cite{kadison_ringrose:I}, \cite{kadison_ringrose:II}, and \cite{Tak}.

Let $H$ be a separable Hilbert space with the inner product $(\cdot,\cdot)$. Denote by $L(H)$ the algebra of bounded linear
operators on $H$. Each pair of nonzero vectors
$\xi,\eta\in H$ determine  semi-norms
$\|A\|_{\xi,\eta}=|(A\xi,\eta)|$ and $\|A\|_\xi=\|A\xi\|$ on $L(H)$.
 In this paper, by an algebra of operators
we always mean a $*$-algebra  containing the identity
operator $I$.

\begin{definition}\label{DefinitionWeakTopology}
 The weak (strong) operator topology on $L(H)$ is the least topology making all the semi-norms $\|\cdot\|_{\xi,\eta}$
($\|\cdot\|_\xi$) continuous.
\end{definition}

For a subset of operators  $\mathcal{S}\subset L(H)$
the {\it commutant} of $\mathcal{S}$ is denoted by
 $$\mathcal{S}'=\{A\in L(H):AB=BA\;\text{for each}\;B\in \mathcal{S}\}.$$
The commutant of a subset of operators is always an algebra closed in the weak and strong
operator topologies. The commutant of $\mathcal{S}'$ is denoted by $\mathcal{S}''$.
The proof of the following von Neumann bicommutant theorem can be found, for example, in
 \cite[Theorem 2.4.11]{bratelli_robinson:1982}.
\begin{theorem}\label{TheoremBicommutant} Let $\mathcal{A}\subset L(H)$
 be a $*$-subalgebra. Then the following conditions are equivalent:

(1) $\mathcal{A}$ is closed in the weak operator topology;

(2) $\mathcal{A}$ is closed in the strong operator topology;

(3) $\mathcal{A}=\mathcal{A}''$.
\end{theorem}

\begin{definition}\label{DefinitionWAlgebra} A $*$-subalgebra
$\mathcal{A}\subset L(H)$ is called a {\it von Neumann (or $W^*$-) algebra}
if it meets the conditions of Theorem \ref{TheoremBicommutant}.\end{definition}

\begin{definition}\label{DefinitionRepresentation}
 A {\it unitary representation} of a group $G$ is a homomorphism
$G\rightarrow U(H)$, where $U(H)$ is the group of unitary operators on a Hilbert
space $H$. If
$G$ is a topological group, the representation is called {\it  continuous} if it is continuous with respect to the weak
operator topology on $L(H)$.
\end{definition}
For a unitary representation $\pi$ of a group $G$ denote by
$\mathcal{M}_\pi$ the $W^{*}$-algebra generated by the operators $\pi(G)$. Two unitary representations $\pi_1$ and $\pi_2$ of the same
group $G$ are
called {\it quasi-equivalent} if there is an isomorphism of von Neumann algebras
$\omega:\mathcal{M}_{\pi_1}\rightarrow \mathcal{M}_{\pi_2}$ such that
$\omega(\pi_1(g))=\pi_2(g)$ for each $g\in G$.

\begin{definition}\label{definitioninitionRepresentations}
A representation $\pi$ of a group $G$ is called a {\it
factor representation} if the algebra $\mathcal{M}_\pi$ is a
factor,  that is  $\mathcal{M}_\pi \cap \mathcal{M}_\pi'=\mathbb{C}I$.
\end{definition}

\begin{definition}\label{definitioninitionCharacter}
A {\it character} of a group $G$ is a function $\chi:G\rightarrow
\mathbb{C}$ satisfying the following properties:
\begin{itemize}
\item[(1)] $\chi(g_1g_2)=\chi(g_2g_1)$ for any $g_1,g_2\in G$;
\item[(2)] the matrix
$\left\{\chi\left(g_ig_j^{-1}\right)\right\}_{i,j=1}^n$ is
nonnegative-definite for any integer $n\geq 1$ and elements $g_1,\ldots,g_n\in G$;
\item[(3)] $\chi(e)=1$, where $e$ is the  identity of $G$.
\end{itemize}

A character $\chi$ is called {\it indecomposable} if it
cannot be represented in the form $\chi=\alpha
\chi_1+(1-\alpha)\chi_2$, where $0<\alpha<1$ and $\chi_1,\chi_2$ are
distinct characters.
\end{definition}

Fix a character $\chi$ of the group
$G$. Consider the space $L_0(G)$ of functions on $G$ with finite supports. Define the bilinear form on
 $L_0(G)$ by
$$B_\chi(f_1,f_2)=\sum\limits_{g_1,g_2\in G}f_1(g_1)
\overline{f_2(g_2)}\chi(g_1g_2^{-1}).$$ It follows from Definition \ref{definitioninitionCharacter}
 that $B_\chi$ is a nonnegative-definite Hermitian form.
Let $\widetilde{H}_\chi$ be the completion of $L_0(G)$ with respect to
$B_\chi$. Set $$H_\chi=\widetilde{H}_\chi/Ker(B_\chi).$$  Then
$H=H_\chi$ is a Hilbert space with the inner product $(\cdot,\cdot)$ induced
by $B_\chi$. Define an
action $\pi$ of $G$ on $L_0(G)$ by the formula
$$(\pi(g)f)(h)=f(g^{-1}h).$$ Observe that the form $B_\chi$ is invariant
under $\pi(G)$. Therefore, $\pi$ defines a unitary representation
of $G$ on $H$, which we  denote by the same symbol $\pi$.  Consider the function
 $\delta\in L_0(G)$ such that $\delta(e) = 1$ and $\delta(g) = 0$ for all $g\neq e$.  Denote by $\xi\in H$ the image of $\delta$ in $H$. It follows that

 $$(\pi(g)\xi,\xi)=B_\chi(\pi(g)\delta,\delta)=\chi(g)\mbox{ for any }g\in G.$$
One can show that the vector $\xi$ is cyclic and separating for the algebra
$\mathcal{M}_{\pi}$. Recall that $\xi$ is {\it separating} if $A\xi = 0$,
$A\in \mathcal M_\pi$, implies $A = 0$. The procedure described above is called  the
Gelfand-Naimark-Siegal (abbreviated ``GNS'') construction. Using the disintegration theorem \cite[Theorem 8.21]{Tak}, one can also establish the following result.

\begin{proposition} The representation $\pi$ constructed above by $\chi$
is a factor representation if and only if the character $\chi$ is indecomposable.
\end{proposition}

If $\chi$ is an indecomposable character, then the constructed von Neumann algebra $\mathcal M_\pi$ is  a { finite-type} factor, i.e.  it is of type $I_n$, $n<\infty$, or of type $II_1$, see  definitions in \cite[Chapter 5]{Tak}.

 Conversely,   consider a finite-type factor representation $\pi$ of a group $G$ on a Hilbert space. Then there is a unique positive linear functional $tr$ on the algebra $\mathcal M_\pi$ with the properties $tr(I) = 1$,  $tr(AB) = tr(BA)$ for all $A,B\in \mathcal M_\pi$, and $tr(A^* A) > 0$ if $A \neq 0$, see \cite[Theorem 8.2.8]{kadison_ringrose:II}.  Applying the GNS construction (similar to the one above) to the algebra $\mathcal M_\pi$ (see \cite[Theorem 4.5.2]{kadison_ringrose:I}), we can find another realization of $\mathcal M_\pi$ on a Hilbert space $H$ such that $tr(A) = (A\xi,\xi)$ for some cyclic and separating vector $\xi\in H$. Hence $\chi(g) = tr(\pi(g)) = (\pi(g)\xi,\xi)$ is an indecomposable character on the group $G$.  We can summarize the discussion above in the following result.

 \begin{theorem}
(1) Indecomposable characters are in one-to-one correspondence with finite-type factor representations.

(2) Two factor representations $\mathcal M_{\pi_1}$ and $\mathcal M_{\pi_1}$ constructed by characters $\chi_1$ and $\chi_2$ are quasi-equivalent if and only if $\chi_1\equiv \chi_2$.

 \end{theorem}

 Throughout the paper, the symbol  $\chi$ will be reserved for a indecomposable character of some full group $G$. We will also use the notation $(\pi,H,\xi)$ to denote  a representation $\pi$, Hilbert space $H$, and a cyclic and separating vector $\xi\in H$ obtained by the GNS-construction for $\chi$.

%\cite{dixmier:1969} Proposition 17.3.4)

The following simple result shows that the continuity of the representation is equivalent to
the continuity of the corresponding character.

\begin{proposition}\label{PropositionAutomaticContinuity}
Let $G$ be a topological group. If a character $\chi:G\rightarrow \mathbb C$
is continuous, then the corresponding  factor representation $\pi: G \rightarrow \mathcal M_\pi$ is continuous.
\end{proposition}
\begin{proof} For any $h_1,h_2\in G$, the function $$f_{h_1,h_2}(g) := (\pi(g)\pi(h_1)\xi,\pi(h_2)\xi ) = \chi(h_2^{-1}gh_1)$$ is continuous in the group topology on $G$. Since $Lin\{\pi(h)\xi : h\in G)\}$ is dense in the Hilbert space $H$, we get that for any $x,y\in H$ the function $f_{x,y}(g) = (\pi(g)x,y)$ is continuous on $G$. This implies that  $\pi: G \rightarrow \mathcal M_\pi$ is continuous. \end{proof}

%%%%%%%%%%%%%%%%%%%%%

%%%%%%%%%%%%%%%%%%%%%%%%%%
%
% Permutation groups. ALGEBRAIC PROPERTIES
%

\section{Algebraic Properties of Full Groups}\label{SectionAlgebraicProperties}  In this section,  we establish several algebraical  results on full groups needed in consecutive sections. Additional algebraic and dynamical properties of full groups can be found in \cite{glasner_weiss:1995}, \cite{giordano_putnam_skau:1999}, \cite{matui:2006}, and \cite{bezuglyi_medynets:2008}. Recall that the symbol $G$ stands for the full group of a Bratteli diagram and $X$ for its path-space.  We always assume that the system $(X,G)$ is minimal.

\begin{definition} (1) For each element $g\in G$, set $supp(g) = \{x\in X : g(x)\neq x\}$ and $Fix(g) = \{x\in X : g(x) = x\}$, termed the {\it support} of $g$ and the {\it set of fixed points of $g$}, respectively. Note that $supp(g)$ and $Fix(g)$ are clopen sets.

(2) We will denote by $\mathcal M(G)$ and $\mathcal E(G)$ the set of all probability  $G$-invariant and probability ergodic   $G$-invariant measures, respectively.

%(3)  For a clopen set $A\subset X$, we  set $F(A) = (\mu(A))_{\mu\in\mathcal E(G)}\in [0,1]^{\mathcal E(G)}$.
\end{definition}

The following lemma is a crucial technical result that shows that if $\mu(A)<\mu(B)$ for all $\mu\in\mathcal M(G)$, then $B$ contains a  subset $G$-equivalent to the set $A$, for the proof see \cite[Lemma 2.5]{glasner_weiss:1995}.

%%%%%%%%

\begin{lemma}\label{LemmaMain} Let $A$ and $B$ be clopen subsets of $X$.
 If  $\mu(A)<\mu(B)$ for all $\mu\in\mathcal M(G)$, then there is an involution $g\in G$ with $g(A)\subset B$ and $supp(g)\subset A\cup B$.
\end{lemma}

The following lemma immediately follows from the minimality of the system $(X,G)$ and the weak-* compactness of the set $\mathcal M(G)$.

 \begin{lemma}For each clopen set $A\subset X$, $A\neq \varnothing$
the number $\delta(A) = \inf\{\mu(A) : \mu\in \mathcal M(G)\}$ is strictly positive.
 \end{lemma}

 For any two elements $g,h\in G$, set $$D(g,h) =\sup_{\mu\in\mathcal E(G)} \mu(\{x\in X : g(x)\neq h(x)\}).$$ Observe that the function $D$ is invariant with respect to the group operations, i.e. $D(g,h) = D(qg,qh) = D(gq,hq) = D(g,h) = D(g^{-1},h^{-1})$ for every $g,h,q\in G$, see, for example, the chapter on the uniform topology in \cite{halmos:book}.  Since for every clopen set $A$ the value  $\mu(A)$ is strictly positive for all $\mu\in\mathcal M(G)$, we get that $D(g,h) = 0$ implies $g = h$. Furhermore, using the same ideas as in  \cite{halmos:book}, one can check that $D$ satisfies the triangle inequality.  This shows that $G$ is a topological group with respect to the metric $D$.  Following the traditions of ergodic theory, we  give the following definition.

 \begin{definition}\label{DefinitionUniformTopology} The topology generated by the metric $D$ is called the {\it uniform topology}.
 \end{definition}

\begin{definition}\label{DefinitionEvenCycles} (1)  We say that elements $g$ and $h$ are {\it $\varepsilon$-conjugate} if there is $q\in G$ with $D(qgq^{-1},h)<\varepsilon$.

(2) We say that an element $g$ {\it consists of even cycles} if the $g$-period of every point $x\in X$ is equal to one or to an even number. The word ``even'' refers to the length of the cycle and not to the parity of a permutation.
\end{definition}

  Given $n\geq 1$ and a clopen set $A$, define a subgroup  $G_n(A)$ as the set of all elements $g\in G_n$ such that $supp(g)\subset A$. Notice that $G_n(A) \subset G_{n+1}(A)$. Set $G(A) = \bigcup_{n\geq 1}G_n(A)$.  We will refer to groups $G(A)$ as to  {\it local subgroups}. The following technical lemmas show  connections between the algebraic structure of the full group and its action on the Bratteli diagram.

%%%%%%%
%
%  Lemma CONJUGATIONS

\begin{lemma}\label{LemmaSequenceInvolutions} Let $G$ be the full group of a Bratteli diagram $B$. Let $A$ be a clopen subset of the path-space $X$. There exists a sequence of involutions $\{h_n\}_{n\geq 1}\subset G(A)$ satisfying the following conditions:
\begin{itemize}

\item[(i)] $h_n$ commutes with every element of $G_n(A)$;

\item[(ii)] for every involution $s\in G_n(A)$ the elements $s h_n$ and $h_n$ are $1/n$-conjugate;

\item[(iii)]  if an element $s\in G_n(A)$ with $supp(s) = A$
consists of even cycles, then the elements $s h_n$ and $s$ are conjugate;

 \item[(iv)] the elements $h_{n_1}h_{n_2}h_{n_3}$ and $h_{n_1}h_{n_2}$, $n_1<n_2<n_3$, are $1/{n_1}$-conjugate;

 \item[(v)] the elements $h_{n_1}$ and $h_{n_2}$, $n_1<n_2$, are $1/{n_1}$-conjugate.

\end{itemize}
All elements implementing the conjugations above are taken from the group $G(A)$.

If the full group $G$ is simple, then,  additionally, all $1/n$-conjugations in the statements  (ii) --- (v) are  conjugations.  For example, the property (ii) states, in this case, that elements $sh_n$ and $h_n$ are, in fact, conjugate.
\end{lemma}
\begin{proof}  To simplify the notation, we will assume that the set $A$ coincides with the entire space $X$ and $G(A) =G$.

 For $m>n$, $v\in V_n$, and $w\in V_m$ denote by
$C_{v,w}$  the set of finite paths between $v$ and $w$.
By minimality, we can telescope the diagram $B$
(denote the new diagram by the same symbol $B$) so that for each level $n$ and vertices
$v\in V_n$ and $w\in V_{n+1}$ we have that $card(C_{v,w}) > n $.

Fix an integer $n\geq 1$. For each pair of vertices $v\in V_n$ and $w\in
V_{n+1}$, find two disjoint subsets
  $C_{v,w}^{(1)},C_{v,w}^{(2)}\subset C_{v,w}$ of equal cardinality such that
$card(C_{v,w}\setminus (C_{v,w}^{(1)}\cup C_{v,w}^{(2)}))\leq 1$.
Let $h_{v,w}:C_{v,w}\to C_{v,w}$ be an involution $\left(h_{v,w}^2=id\right)$, which maps $C_{v,w}^{(1)}$ onto $C_{v,w}^{(2)}$. Denote the
elements $h_n\in G_{n+1}$ as follows \begin{eqnarray}
\label{Equationhn}h_n(e_1,\ldots,
e_n,e_{n+1},e_{n+2}\ldots)=(e_1,\ldots,
e_n,h_{v,w}(e_{n+1}),e_{n+2},\ldots),
\end{eqnarray}  where $(e_1,e_2,\ldots)$ is any infinite path and
$v=s(e_{n+1}),w=r(e_{n+1})$. We will show that the sequences $\{h_n\}_{n\geq 1}$ satisfies all the properties of the lemma. The construction of
$h_n$ implies that the element $h_n$ commutes with all elements of $G_n$. This  proves the property (i).

   Recall that
$G_n = \prod\limits_{v\in V_n}G_n^{(v)}$ and $G_n^{(v)}$ is isomorphic to
the group of permutations $S\left(M_n^{(v)}\right)$ (see Subsection
\ref{SubsectionInductiveLimits}). Observe that two elements $g_1,g_2\in G_n$
 are conjugate if and only if for each $v\in V_n$ the corresponding
permutations from $S\left(M_n^{(v)}\right)$ are conjugate. In particular,
 if $g_1$ and $g_2$ are involutions, they are conjugate if and only if
for each $v\in V_n$ the elements  $g_1$
 and $g_2$ have the same number of  fixed paths in $M_n^{(v)}$.

Fix an involution $s\in G_n$. Then the element $sh_n$ fixes a  path $(e_1,\ldots,e_{n+1})$
 if and only if the permutation $s$ fixes the path $(e_1,\ldots,e_n)$  and the involution $h_{v,w}$ fixes the edge
 $e_{n+1}$, where $v=s(e_n),w=r(e_{n+1})$.
It follows from the definition  of $h_n$ that for each $w\in V_{n+1}$  the number of fixed paths
in $M_{n+1}^{(w)}$ (for both  $h_n$ and $sh_n$) is at most $card(M_{n+1}^{(w)})/n$. Therefore, we can find
an involution $h'_n\in G_{n+1}$ such that for each $w\in V_{n+1}$

\begin{enumerate}

\item[(1)]  $h_n'$ differs from
$h_n$ on $M_{n+1}^{(w)}$ on at most $card(M_{n+1}^{(w)})/n$ paths;

\item[(2)] $h_n'$ has the same number of fixed paths in $M_{n+1}^{(w)}$
as $sh_n$.
\end{enumerate}

 Then $h'_n$ is conjugate to $sh_n$. Fix a $G$-invariant measure $\mu$ on $X$. For a vertex $w\in W_{n+1}$, denote by $X_w$  the set of all paths
 coming though the vertex $w\in W_{n+1}$. Set $\alpha_w = \mu(X_w)$. Observe that $\sum_{w\in W_{n+1} }\alpha_w = 1$ and $\mu([x_w]) = \alpha/card(M_{n+1}^{(w)})$ for any finite path from the root vertex $v_0$ to the vertex $w$. Here $[x_w]$ is the cylinder set of infinite paths coinciding with $x_w$. Therefore,
 \begin{eqnarray*}\mu(\{x\in X:h_n(x)\neq h'_n(x)\}) & = & \sum_{w\in W_{n+1} } \mu(\{x\in X_w:h_n(x)\neq h'_n(x)\}) \\
 & \leq & \sum_{w\in W_{n+1} }\frac{\alpha_w}{card(M_{n+1}^{(w)})} \cdot \frac{card(M_{n+1}^{(w)})}{n}\\
 & = & 1/n.
 \end{eqnarray*}
This proves (ii).

Now let $s\in G_n$ consist of even cycles and $supp(s)=X$. Let $x=(x_1,\ldots,x_{n+1})$
be a finite path and $x'=(x_1,\ldots,x_n)$. Then, by definition of $h_n$, we have that
$$(sh_n)^l(x)=\left\{\begin{array}{ll}(s^l(x'),h_{v,w}(x_{n+1})),
&\text{if}\;l\;\text{is odd},\\
(s^l(x'),x_{n+1}),
&\text{if}\;l\;\text{is even},
\end{array}\right. $$ where $v=s(x_{n+1}),w=r(x_{n+1})$. Observe that the length
of the orbit of $x$ under $s$ is
even. It follows that the orbits of $x$ under the actions of $s$ and $sh_n$ have
the same lengths. Therefore, for any vertex $w\in V_{n+1}$ the elements $s$ and $sh_n$ induce
conjugate permutations on $M_{n+1}^{(w)}$, which proves the property (iii).

The properties (iv) and (v) can be established similarly to
(ii). We leave the details to the reader.
\end{proof}

% Lemma Second Technical

\begin{lemma}\label{lemmaExistenceSI}
Let $s\in G$ and $A = supp(s)$. Then for any $r\in \mathbb{N}$
and any $\varepsilon>0$ there exist elements
$s_1,\ldots,s_{2^r}\in G$ satisfying the following conditions:

\begin{itemize}
\item[(1)] each element $s_i$ is conjugate to $s$;
\item[(2)] $supp(s_i)=A$ and $\mu(A\setminus supp\left(s_is_j^{-1}\right))<\varepsilon$ for any   $1\leq i\neq j\leq 2^r$ and $\mu\in \mathcal M(G)$;
\item[(3)] the element $s_is_j^{-1}$, for any $i\neq j$, consists of even cycles (Definition \ref{DefinitionEvenCycles}).
\end{itemize}
\end{lemma}
\begin{proof}  Denote by $A_p$ the set of all points $x\in X$
 such that $s^p(x) =x$ and $s^l(x)\neq x$ for $1\leq l\leq p-1$.
Then $X = \bigsqcup A_p$ is a clopen partition.
Clearly, it is enough to prove  the lemma for each restriction
$s|_{A_p}$ and ensure that the elements $s_i^{(p)}$ are taken from $G(A_p)$. Then  $s_i = \prod_p s_i^{(p)}$, $i=1,\ldots, 2^r$,  will be the desired elements. So, without loss of generality, we will assume that there exists an integer $p\geq 2$ such that $X = A_p$.

We will use the notation $(k_0,k_1\ldots,k_{l-1})$ to  denote
 the cyclic permutation of  symbols $\{k_0,\ldots, k_{l-1}\}$. For example, $(k,l)$  will stand for the
transposition of $k$ and $l$.  Set also $M_n = \{0,1\ldots, n-1\}$. The symmetric group on $M_n$ will be denoted by $S(M_n)$.
 The  proof of the lemma will be based on the following combinatorial observation.

\bigskip\noindent {\bf Claim 1.} {\it Let $p\geq 2$ be an integer. Set $m=2p$ if $p$ is even and $m = 2p-2$ if $p$ is odd.  There exists  permutations $h_0,h_1\in S(M_{m})$ such that (i) each  $h_i$ is a cycle of length $p$; (ii) each independent cycle of $h_0 h_1^{-1}$ is of even length; and (iii)
 $card(supp(h_0 h_1^{-1})) = m$.}

{\it Proof of the claim.} If $p$ is even, then set
$$h_0=(0,1,\ldots,p-1) \mbox{ and } h_1 = (p,p+1,\ldots,2p-1).$$ The claim easily follows from the definition of $h_0$ and $h_1$.

If $p$ is odd, then define $h_0$ and $h_1$ as
$$h_0=(0,1,\ldots,p-1) = (0,1)\cdots (p-2,p-1)$$
and $$h_1 = (2p-3,2p-4,\ldots,p-2) = (2p-3,2p-4) \cdots (p-1,p-2).$$
Therefore,
\begin{eqnarray*}
h_0h_1^{-1} & = & (0,1)(1,2)\cdots (p-3,p-2)(p-2,p-1) \\
& \times & (p-2,p-1)(p-1,p) \cdots (2p-2,2p-3) \\
& = & (0,1,\ldots,p-2)\times (p-1,p,\ldots,2p-3).
\end{eqnarray*}
Thus, $h_0h_1^{-1}$
  is a product of two
cycles of length $p-1$.

Now the rest of the properties  immediately follows from this construction.
{\it This proves the claim.}

\bigskip Recall that $V_n$ stands for the set of vertices of the $n$-th level.
For each pair of vertices $v\in V_n$ and $,w\in
V_{n+1}$ denote by $E_{v,w}^{(n+1)}$ the set of all edges between $v$ and $w$.
Denote by $G_{n,n+1}$ the subgroup of all elements $g\in G$ such that $g$ changes only the $n+1$-st edge of $x$, $x\in X$, i.e. the edge
 between $n$-th and $n+1$-st levels. Clearly, $G_{n,n+1}\subset G_{n+1}$.

 Telescope the diagram $B$ so that $s\in G_1$ and for any $n\geq 1$,
any pair of vertices $v\in V_n,w\in V_{n+1}$ we have that
\begin{eqnarray}\label{Equation1LemmaTech2}
\left|E_{(v,w)}^{(n+1)}\right|>\frac{2p}{\varepsilon}.\end{eqnarray}

\bigskip\noindent{\bf Claim 2.} {\it Each group $G_{n,n+1}$, $n\geq 1$, contains two elements $g_0,g_1$ that satisfy the following conditions
\begin{enumerate}
\item[(i)]  the length of every nontrivial cycle of $g_0$ and $g_1$ is $p$;
\item[(ii)] all nontrivial cycles of $g_0g_1^{-1}$ have even lengths;
\item[(iii)] $\mu(supp(g_0g_1^{-1}))>1-\varepsilon$ for each  measure $\mu\in\mathcal M(G)$.
\end{enumerate}
}

{\it Proof of the claim.}
(a) For every pair of vertices $v\in V_n$ and $w\in V_{n+1}$, consider the set  $E_{v,w}^{(n+1)}$. Set $m = 2p$ if $p$ is even and $m = 2p-2$ if $p$ is odd. Find non-negative integers $t$ and $q$ such that $card(E_{v,w}^{(n+1)}) = m t + q,$ where
 $q\in\{0,1,\ldots,m-1\}$. Note that the numbers $t$ and $q$ depend
on the vertices $v$ and $w$. The arguments below shall be
independently applied  to each pair of vertices.
Leave $q$ edges of the set $E_{v,w}^{(n+1)}$ out and partition
the rest of the edges into $t$ disjoint families $F_0,\ldots, F_{t-1}$
 with $card(F_i) = m$. For each $i$, let $h_0^{(i)}$ and $h_1^{(i)}$ be
permutations as in Claim 1 above. Note that the definition of $h_0^{(i)}$
and $h_1^{(i)}$ depends on the parity of the number $p$ (as well as on the  vertices $v$ and $w$). Set $h^{(j)}_{v,w}  = h_j^{(0)}\cdots h_j^{(t-1)}$, $j=0,1$.

 Thus, we have constructed permutations  $h^{(0)}_{v,w}$ and $h^{(1)}_{v,w}$ of the set $E_{v,w}^{(n+1)}$ such that each permutation $h^{(j)}_{v,w}$, $j=0,1$, consists of disjoint cycles of length $p$ and
 \begin{equation}\label{Equation2LemmaTech2} card\left(h^{(0)}_{v,w} (h^{(1)}_{v,w})^{-1}\right) = t\cdot m.\end{equation}

 Define a homeomorphism $g_i:X\rightarrow X$  by
 \begin{equation}\label{Equation3LemmaTech2}g_i(x_1,\ldots x_{n-1},x_n,\ldots)
= (x_1,\ldots x_{n-1},h_{v,w}^{(i)}(x_n),\ldots)\mbox{ if }x_n\in E_{(v,w)}^{(n+1)}.\end{equation}

 Clearly, $g_0$ and $g_1$ are cycles of length $p$ belonging to $G_{n,n+1}$.
It follows from Equations (\ref{Equation1LemmaTech2})-(\ref{Equation3LemmaTech2})
that for each pair of vertices $v\in V_n,w\in V_{n+1}$
the homeomorphism $g_0g_1^{-1}$ acts on all, but less then $\varepsilon
\left|E_{(v,w)}^{(n+1)}\right|$
 edges from $E_{(v,w)}^{(n+1)}$.
Therefore, $$\mu(supp(g_0g_1^{-1}))>1-\varepsilon\mbox{ for every invariant measure }\mu.$$
This establishes the condition (iii) of the claim. The rest of the conditions follows easily from the definition of $g_0$ and $g_1$. {\it This proves the claim.}

\bigskip For each $n\geq 1$, take homeomorphisms $g_{n+1}^{(0)}, g_{n+1}^{(1)} \in G_{n,n+1}$ satisfying all conditions of Claim 2 above. For each tuple $a=(a_1,a_2,\ldots,a_r)$, $a_i\in
\{0,1\}$, define a homeomorphism $s_a^{(i)}\in G_{1+r}$ by

$$s_a(x) =
sg^{(a_1)}_2 \cdot g^{(a_2)}_3\cdots g^{(a_r)}_{r+1}(x),$$
where $x \in X$.  Recall that $s\in G_1$.  Since all homeomorphisms $\{g^{(a_i)}_{i+1}\}$ act on different levels (starting with $n\geq 2$) and consist of cycles of length $p$, we conclude that  $s_a$ is conjugate to $s$ for every $a\in \{0,1\}^r$. Moreover, for distinct $a, b\in
\{0,1\}^r$ we have that
$$
s_as_b^{-1}(x) =  g^{(a_1)}_2 (g^{(b_1)}_2)^{-1}\cdot \cdots g^{(a_r)}_{r+1} (g^{(b_r)}_{r+1})^{-1}(x),\; x\in X.$$ Choose $i=1,\ldots, r$ such that $a_i\neq b_i$. Then $supp(s_as_b^{-1})\supset supp(g^{(a_i)}_{i+1} (g^{(b_i)}_{i+1})^{-1}) $ and $\mu(supp(s_as_b^{-1}))>1-\varepsilon$ for every $\mu\in\mathcal M(G)$.

  To complete the proof of the lemma, we observe only that the homeomorpism $s_as_b^{-1}$ consists of cycles of length $p$ if  $p$ is even and of length $p-1$ if $p$ is odd.
\end{proof}

%%%%%%%%%%%%%%%%%%%
%
%
% ORTHOGONAL PROJECTOR

\section{Orthogonal Projections}\label{OrthogonalProjections}

In the section we define a family of orthogonal projections,
which, in some sense, behave as operators of multiplication by characteristic functions.
Here we also establish the first main result of the paper
(Theorem \ref{TheoremTracePreliminaries}).
 Fix an indecomposable character $\chi$ of the group $G$ and the
corresponding factor representation $\pi: G\rightarrow \mathcal M_\pi$,
see Section \ref{SubsectionCharacters}. The symbol $H$ stands
for the Hilbert space on which the algebra $\mathcal M_\pi$ acts and
 $\xi$ stands for the cyclic and separating vector for the algebra $\mathcal M_\pi$.
 Recall that $tr(Q) = (Q\xi,\xi)$ for all $Q\in \mathcal M_\pi$ and $\chi(g) = tr(\pi(g))$ for all $g\in G$.

\begin{definition}\label{DefinitionProjectors} Fix a clopen set $A\subset X$. (1) % For level $n\geq 1$,  define the  average operator by $$S_n^A=\frac{1}{|G_n(A)|}\sum_{g\in G_n(A)} \pi(g).$$
%
%(2)
 Denote by $P^A$ the orthogonal projector onto the space $$H^A = \{h\in H : \pi(g)h = h\mbox{ for all }g\in G(A)\}.$$
\end{definition}

\begin{remark} \label{RemarkPropertiesProjectors}

(1) Since the character $\chi$ is indecomposable, the existence of a non-zero invariant vector for the group $\pi(G)$ implies that the representation $\pi$ is trivial, i.e. $\pi(g) = I$ is the identity operator for every $g\in G$. This implies that  $\chi(g)=1$ for all
$g$. If $\chi
\neq 1$, then there are no invariant vectors, which implies that $P^X = 0$.

(2) We will use the convention that $P^\emptyset = I$ is the identity operator.

(3) Notice that it is not clear a priori that $P^A\neq 0$ for at least one clopen set
$A\neq \emptyset$.

(4) If $g\in G(A)$, then $\pi(g)P^A = P^A$.

(5)  It will follow from Proposition \ref{PropositionWeakLimitsPA} that the projection $P^A$ belongs to $\mathcal M_\pi$ the von Neumann algebra generated by $\pi(G)$.
\end{remark}

 We recall that the projector $P_1$ {\it dominates} $P_2$, in symbols $P_1\geq P_2$,
if $P_2(H)\subset P_1(H)$. This, in particular, implies that $tr(P_1)\geq tr(P_2)$.
 Recall that when we talk about limits in $\mathcal M_\pi$, we always mean {\it weak limits}.

\begin{proposition}\label{PropositionWeakLimitsPA} Let $B$ be a Bratteli diagram and $G$ the associated full group. Assume that either
the group $G$ is simple or the representation $\pi: G\rightarrow \mathcal M_\pi$
is continuous. Let $A$ be a clopen subset of the path-space $X$ and $\{h_n\}_{n\geq 1}$
be a sequence of involutions satisfying the properties of Lemma
\ref{LemmaSequenceInvolutions} for the set $A$.  Then
$$P^A =  \lim\limits_{m\to\infty }\pi(h_n).$$
\end{proposition}
\begin{proof}  We will simultaneously consider the cases of  simple groups
(even diagrams, see Proposition \ref{PropositionSimpleDiagrams}) and continuous representations. Note that if the
representation $\pi$ is continuous, then the corresponding character
$\chi$ is also continuous.

Consider a sequence  $\{h_n\}_{n\geq 1}$ satisfying all the statements
  of Lemma \ref{LemmaSequenceInvolutions} and  the corresponding sequence of unitary operators $\{\pi(h_n)\}_{n\geq 1}$.  Passing to a subsequence (we will drop the extra subindex), we can assume that $\{\pi(h_n)\}$ weakly converges to some operator $Q\in\mathcal \mathcal M_\pi$.

 {\it Claim 1.} {\it $Q$ is an orthogonal projector.}
 Since all elements $\{h_n\}$ are involutions, we obtain that $\pi(h_n)^* = \pi(h_n^{-1}) = \pi(h_n)$. Hence $Q = Q^*$ is a self-adjoint operator.

  Fix a triple  $n_1<n_2<n_3$. Lemma  \ref{LemmaSequenceInvolutions}(iv) implies that $h_{n_1}h_{n_2}h_{n_3}$ and $h_{n_1}h_{n_2}$ are either conjugate
  (for simple groups) or $1/n_1$-conjugate (for continuous reprsentations),
$n_1<n_2<n_3$. In either case, we have
  $$|tr(\pi(h_{n_1})Q^2) - tr(\pi(h_{n_1})Q)| =
\lim\limits_{n_2\to\infty}\lim\limits_{n_3\to\infty}|
\chi(h_{n_1}h_{n_2}h_{n_3}) - \chi(h_{n_1}h_{n_2})|.$$ Therefore
$$|tr(\pi(h_{n_1})Q^2) - tr(\pi(h_{n_1})Q)|\to 0\mbox{ as }n_1\to\infty.$$
    This implies that  $tr(Q^3) = tr(Q^2)$, i.e.
 \begin{equation}\label{EqnTripple}(Q^3\xi,\xi) = (Q^2\xi,\xi) = (Q\xi,Q\xi) =
||Q\xi||^2.\end{equation}
  Using the Cauchy-Schwarz inequality, we get that $$||Q\xi||^2 =
 (Q^3\xi,\xi) = (Q^2\xi,Q\xi) \leq ||Q\xi||\cdot ||Q^2\xi||
\leq ||Q\xi||\cdot ||Q\xi|| = ||Q\xi||^2,$$ i.e., $(Q^2\xi,Q\xi) =
||Q\xi||\cdot ||Q^2\xi||$. Therefore, the vectors
$Q\xi$  and $Q^2\xi$ are linearly dependent.
Since the vector $\xi$ is separating for the algebra $\mathcal M_\pi$,
we get that $Q^2 = cQ$ for some $c\in \mathbb C$. Using (\ref{EqnTripple}), we see that $||Q\xi||^2 = (Q^2\xi,Q\xi) = c(Q\xi,Q\xi) = c||Q\xi||^2$. This implies that $c = 1$ and $Q$ is a projector.

  {\it Claim 2.} {\it $\pi(s)Q = Q$ for every $s\in G(A)$}. Since every element of $G(A)$ can be represented as a product of involutions, it sufficies to establish the claim for involutions only.  Fix an involution $s\in G(A)$. For all $n$ large enough, the element $sh_n$ and $h_n$ are either conjugate (for simple groups) or $1/n$-conjugate (for continuous representations), see Lemma  \ref{LemmaSequenceInvolutions}(ii). It follows that $tr(\pi(s)Q)   = tr(Q)$, i.e. $$(\pi(s)Q\xi,\xi) = tr(\pi(s)Q)   = tr(Q) = (Q\xi,\xi) = (Q^2\xi,\xi) =(Q\xi,Q\xi) = ||Q\xi||^2.$$

Using the Cauchy-Schwarz inequality, we get that
 $$||Q\xi||^2 = (\pi(s)Q\xi,\xi)=(\pi(s)Q\xi,Q\xi)\leq
\|\pi(s)Q\xi\|\cdot \|Q\xi\|
\leq\|Q\xi\|^2.$$ Since we have the equality
here, the vectors $\pi(s)Q\xi$ and $Q\xi$ are linearly dependent. It is easy to check that, in fact,  $\pi(s)Q\xi = Q\xi$. Since the vector $\xi$ is separating, we conlcude that $\pi(s)Q=Q$.

 {\it Claim 3. $Q = P^A$.}  Since  $h_n\in G(A)$, we have that $\pi(h_n)P^A = P^A$.
Therefore, $QP^A = P^A$, i.e. $P^A\leq Q$. On the other hand,
 Claim 2 above implies that  every vector from the subspace $QH$ is a fixed vector for
$\pi(s)$,
 for any $s\in G(A)$. This implies that $QH\subset P^AH$
(see the definition of $P^A$), i.e.  $Q\leq P^A$.   It follows that
 $Q =P^A$. Thus, the sequence $\pi(h_n)$ has a unique weak limit point, which
coincides with $P^A$. Therefore, $$P^A =  \lim\limits_{m\to\infty }\pi(h_n).$$
 \end{proof}

%%%%%%%%%%%%%%%%%%%%%%%%%%%%%%%%%%%%%%

The next proposition shows that the projectors $\{P^A\}$ form an Abelian semigroup.

\begin{proposition}\label{PropositionPropertiesProj} Let $B$ be a Bratteli diagram and $G$ the associated full group. Assume that either $G$ is simple or the representation of $G$ is continuous.  Let  $A_1$ and $A_2$ be clopen subsets of $X$.

(1) If $A_1 \subset A_2$, then $P^{A_1}\geq P^{A_2}$ and $tr(P^{A_1})\geq tr(P^{A_2})$.

(2) For any $g\in G$, we have that $\pi(g)P^{A_1} \pi(g^{-1}) = P^{g(A_1)}$ and $tr(P^{A_1}) = tr(P^{g(A_1)})$.

(3) $P^{A_1}P^{A_2} = P^{A_1\cup A_2}$.
\end{proposition}
\begin{proof}  (1) Note that if $A_1\subset A_2$, then $G(A_1)\subset G(A_2)$. It follows that $H^{A_2}\subset H^{A_1}$ and $P^{A_2}\leq P^{A_1}$.

(2) Fix an element $g\in G$. Clearly, $\pi(g)P^{A_1} \pi(g^{-1})$ is a projector onto $\pi(g)H^{A_1}$.  Since $\pi(g)H^{A_1}$ consists of vectors invariant under $\pi(g)\pi(q)\pi(g^{-1})$, $q\in G(A_1)$, and $g G(A_1)g^{-1} = G(g(A_1))$, we conclude that $\pi(g)P^{A_1} \pi(g^{-1}) = P^{g(A_1)}$.

(3-i) First assume that clopen sets $A_1$ and $A_2$ are disjoint.    Find the sequences $\{h_n^{(i)}\}\subset G(A_i)$ of involutions satisfying all the properties of  Lemma \ref{LemmaSequenceInvolutions} for the set $A_i$, $i=1,2$. Observe that the sequence (or some subsequence) $h_n = h_n^{(1)}h_n^{(2)}$, $n\geq 1$, also meet all the properties  of Lemma \ref{LemmaSequenceInvolutions} for the set $A_1\cup A_2$. Thus, passing to subsequences and using Proposition \ref{PropositionWeakLimitsPA}, we conlude that
\begin{eqnarray*}P^{A_1}P^{A_2}  & = &\lim_{n\to\infty} P^{A_1}\pi(h_n^{(2)}) = \lim_{n\to\infty} P^{A_1}\pi(h_n^{(1)}h_n^{(2)}) = P^{A_1}P^{A_1\cup A_2} \\
& = & \lim_{n\to\infty} \pi(h_n^{(1)})P^{A_1\cup A_2} = P^{A_1\cup A_2}.\end{eqnarray*}

(3-ii) Now consider abritrary clopen sets $A_1$ and $A_2$. Set $C = A_1\cap A_2$. It follows from (3-i) that $$P^{A_1\setminus C} P^C P^{A_2\setminus C} = P^{A_1} P^{A_2\setminus C} = P^{A_1\cup A_2}.$$ This completes the proof.
\end{proof}

%%%%%%%%%%%%%%%%%%%%%%%%5
% characters and projectors

Let $A$ and $B$ be clopen sets. Set $d(A,B) = \sup\{\mu(A\triangle B) : \mu\in\mathcal M(G)\}$.
Observe that $d(A,B)\leq d(A,C) + d(C,B)$ for any clopen set $C$.  The following lemma shows that the trace of the projector $P^A$ depends ``continuously'' on a clopen set $A$.  We would like to emphasize that all sets in this lemma are assumed to be non-empty.

\begin{lemma}\label{LemmaContinuityTraceOnPA}  For  every clopen
set $A$ and $\varepsilon>0$, there exists $\delta>0$ such that if   a clopen set $B$ satisfies $d(A,B) < \delta$,
then $|tr(P^A) - tr(P^B)|<\varepsilon$.
 \end{lemma}
\begin{proof}  Choose an integer $m$ such that $2/m<\varepsilon$.
Setting $C = A\cap B$, it suffices to show that there exists $\delta>0$
such that  if $d(Z,C)<\delta$, for $Z = A$ and $Z=B$, then $|tr(P^Z) - tr(P^C)|<1/m$. We will only show how to choose $\delta$ for $Z = A$. The proof for the set $B$ is analogous.

 Set $$\delta  = \frac{\inf\{\mu(A) : \mu\in \mathcal M(G)\} }{2m}.$$
It follows that if $C\subset A$ with $d(A,C)<\delta$, then $\mu(C)\geq \mu(A)/2$ and
 $$\mu(A\setminus C) < \delta \leq \frac{\mu(A)}{2 m}\leq \frac{\mu(C)}{m}$$
for every measure $\mu\in\mathcal M(G)$.
Therefore, using Lemma \ref{LemmaMain}, we can find involutions
 $s_j\in G(A)$ with $s_i(A\setminus C)\subset C$, $j=1,\ldots,m$, and
$$s_j(A\setminus C)\cap s_l(A\setminus C) = \emptyset \mbox{ whenever }j\neq l.$$
This implies that  $s_j(C)\cup s_l(C) = A$ for  $j\neq l$. Thus, by Proposition \ref{PropositionPropertiesProj}, we obtain that ($j\neq l$)
$$\begin{array}{ll}\displaystyle \left(P^{s_j(C)} - P^A\right)\left(P^{s_l(C)} - P^A\right)  \\
= P^{s_j(C)}P^{s_l(C)} - P^{s_j(C)}P^A   - P^AP^{s_l(C)} +
 P^AP^A   \\
= P^A - P^A - P^A + P^A =0.
\end{array}$$
Therefore, $\left\{P^{s_j(C)} - P^A\right\}_{j=1}^m$ is a family of mutually orthogonal projectors. It follows that
$\sum\limits_{j=1}^m\left(P^{s_j(C)} - P^A\right)$ is a projector. Hence,
$$1\geq \sum\limits_{j=1}^m tr\left(P^{s_j(C)} - P^A\right) = \sum\limits_{j=1}^m tr\left(\pi(s_j)(P^{C}-P^{A})\pi(s_j^{-1})\right) = m\cdot tr\left(P^{C}-P^{A}\right).$$ It follows that
$0\leq tr\left(P^C-P^{A}\right)\le 1/m$.
\end{proof}

Now we are ready to establish the first main result of the paper.
% The first main result
%
\begin{theorem}\label{TheoremCharactersProjectors} Let $B$ be a Bratteli
diagram and $G$ the associated full group. Assume that either
 the group $G$ is simple or  the representation
$\pi: G\rightarrow \mathcal \mathcal M_\pi$ is continuous with respect
 to the uniform topology on $G$.  Let $\chi$ be the character corresponding to the representation $\pi$.    Then  $$\chi(g) = tr(P^{supp(g)})\mbox{ for every }g\in G.$$
\end{theorem}
\begin{proof} (1) First of all, we will  prove the result when
$g$ consists of even cycles (see Definition \ref{DefinitionEvenCycles}).
 Take a sequence $\{h_n\}_{n\geq 1}$ satisfying the properties of
Lemma \ref{LemmaSequenceInvolutions} for the set $A = supp(g)$.
It follows from Lemma \ref{LemmaSequenceInvolutions}(iii) that the elements $gh_n$ and
$g$ are conjugate, assuming that $n$ is large enough.
Then $\chi(gh_n)=\chi(g)$.
 Applying Proposition \ref{PropositionWeakLimitsPA} and Remark
\ref{RemarkPropertiesProjectors}(4), we get that
$$\chi(g) = \lim_{n\to\infty} \chi(gh_n) = \lim_{n\to\infty}
(\pi(g)\pi(h_n)\xi,\xi) = (\pi(g)P^A\xi,\xi) = (P^A\xi,\xi) = tr(P^A).$$

(2) Now, consider the case when $g$ is an arbitrary non-identity element of $G$. Set $A = supp(g)$. Fix an arbitrary integer  $r\geq 1$ and arbitrary $\varepsilon>0$.  Using Lemma \ref{LemmaContinuityTraceOnPA}, choose $\delta>0$ such that $|tr(P^A) - tr(P^B)|<\varepsilon / r^2$ for any clopen set $B\subset A$ with $d(A,B)< \delta$. By Lemma \ref{lemmaExistenceSI}
there exist elements $s_1,\ldots,s_r\in G$ such that
\begin{itemize}
\item[(i)] each element $s_i$ is conjugate to $g$;
\item[(ii)] $A = supp(s_i)$ and $d(supp\left(s_is_j^{-1}\right),A)<\delta$ for any $i\neq j$;
\item[(iii)]  the element $s_is_j^{-1}$
 consists of even cycles, when  $i\neq j$.
\end{itemize}
  Consider the
family of vectors $\eta_i=\left(\pi(s_i)-P^A\right)\xi$. Put also
$\eta=\left(\pi(g)-P^A\right)\xi$.

 Since $s_i\in G(A)$, we get that  $\pi(s_i)P^A=P^A$.
Thus,  for any
$i\neq j$ we have that
\begin{eqnarray*}
|(\eta_i,\eta_j) | & = & \left | \left((\pi(s_i)-P^A)\xi,(\pi(s_j)-P^A)\xi\right) \right| \\
& = & \left|   (\pi(s_i)\xi,\pi(s_j)\xi) - (\pi(s_i)\xi,P^A\xi) - (P^A\xi,\pi(s_j)\xi) + (P^A\xi,P^A\xi)    \right| \\
& = &
\left |\chi\left(s_is_j^{-1}\right)-tr\left(P^A\right) - tr\left(P^A\right)  + tr\left(P^A\right) \right| \\
& = &
\left|tr(P^{supp(s_is_j^{-1})})-tr(P^A)\right| \\
& \leq & \varepsilon/r^2.
\end{eqnarray*}
Using similar arguments, we obtain that   $\|\eta_i\|^2 = 1-tr\left(P^A\right) = \|\eta\|^2$ for every $i = 1,\ldots, r$.  Since $s_i$ is conjugate to $g$, we see that
$(\eta_i,\xi)=(\eta,\xi)$ for every $i$. It follows that
$$
|(\eta,\xi)|=\frac{1}{r}\left|\left(\sum\limits_{i=1}^r\eta_i,\xi\right)\right|\leqslant
\frac{1}{r}\left\|\sum\limits_{i=1}^r\eta_i\right\|=\frac{1}{r}
\sqrt{\sum\limits_{i,j=1}^r(\eta_i,\eta_j)}\le
\frac{\sqrt{r\|\eta\| +\varepsilon}}{ r}.
$$
Since $r$ and $\varepsilon$ are chosen arbitrarily, we conclude that  $(\eta,\xi)=0$. It
follows that $\chi(g)=tr\left(P^A\right)$.
\end{proof}

%%%%%%%%%%%%%%%%%%%%%%%%%%%%%%%%%%%%
%%%%%%%%%%%%%%%%%%%%%%%%%%%%%%%%%
%
%
% Formula Of The Character.

\section{Formula for the Character}\label{SectionFormulaCharacter}
In this section we prove that each indecomposable character $\chi$ of the group $G$ is defined as $\chi(g) = \mu_1(Fix(g))^{\alpha_1}\cdots \mu_k(Fix(g))^{\alpha_k}$, where $\mu_1,\ldots, \mu_k$ are ergodic measures and $\alpha_1,\ldots,\alpha_k\in \{0,1,\ldots,\infty\}$.  {\it From now on, we assume that the set of ergodic measures $\mathcal E(G)$ is finite. Enumerate all ergodic measures of $(X,G)$ as  $\{\mu_1,\ldots,\mu_k\}$.}

%%%%%%%%%%%%%%%%%%%%%%%%%%%%%%%%%%%%%5
%
%
%

\begin{lemma}\label{LemmaDenseSubset} For any sequences $(x_1,\ldots,x_k)\in [0,1)^k$ and $(y_1,\ldots,y_k)\in (0,1]^k$, any $\varepsilon >0$, and any $n\geq 1$, we can find  clopen $G_n$-invariant sets $A$ and $B$ such that $x_i \leq \mu_i(A) < x_i + \varepsilon$ and $y_i - \varepsilon < \mu_i(B) \leq y_i$ for all $i=1,\ldots, k$.
\end{lemma}
\begin{proof}  (1)  Fix the subgroup $G_n$. Since all ergodic measures are mutually singular, we can find a disjoint family of $G$-invariant Borel sets $B_1,\ldots, B_k$ with $\mu_i(B_i)= 1$, for $i=1,\ldots, k$. For every $i=1,\ldots, k$, choose a closed set $C_i\subset B_i$ with $$x_i\leq \mu_i([C_i]_n)<x_i+ \varepsilon/3,\mbox{ where }[C_i]_n := \bigcup_{g\in G_n}g(C_i).$$
Notice that each set $[C_i]_n$ is closed and $G_n$-invariant.

Since $[C_i]_n\cap [C_j]_n = \emptyset$ for $i\neq j$, we can find a disjoint family of clopen sets $\{A_i\}_{i=1}^k$ such that $[C_i]_n\subset A_i$, $\mu_i(A_i) < \mu_i([C_i]_n) + \varepsilon/3$ and $\mu_i(\bigcup_{j\neq i}A_j)<\varepsilon/3$, $i=1,\ldots,k$. Setting $$A = \bigcup_{i=1}^k \bigcap_{g\in G_n}g(A_i),$$ we get that $A$ is a clopen $G_n$-invariant set that gives the desired approximation.

(2) Consider $(y_1,\ldots,y_k)\in (0,1]^k$. Define a sequence $x = (1-y_1,\ldots, 1-y_k)\in [0,1)^k$. Find a clopen $G_n$-invariant set $A$ as in the part (1). Set $B = X\setminus A$.
\end{proof}

Define a map $F$ from the family of clopen sets into $[0,1]^k$ by  $$F(A) = (\mu_1(A),\ldots,  \mu_k(A)),$$
 where $A$  is a clopen set. Set also  $$M = \{F(A) :
A \mbox{ is a proper clopen set}\}\subset (0,1)^k.$$ Equip $\mathbb{R}^k$ with the $\sup$-norm $\|c\|=\sup\limits_{i=1,\ldots,k}\|c_i\|.$

\begin{lemma}\label{LemmaContinuityofPhi}  For  every $c\in [0,1)^k$
 and $\varepsilon>0$, there exists $\delta>0$ such that if non-empty clopen sets $A$ and $B$
satisfy the conditions
 $\|F(A)-c\|<\delta/2$ and $\|F(B)-c\|<\delta/2$,
then $|tr(P^A) - tr(P^B)|<\varepsilon$.
 \end{lemma}
\begin{proof} %We use the same arguments that we used in the proof of Lemma
% \ref{LemmaContinuityTraceOnPA}.
Choose an integer $m$ such that $2/m<\varepsilon$.
Set $$\delta  = \frac{\min\{c_i : i=1,\ldots,k\}}{3m}.$$

Let $A$ and $B$ be clopen sets satisfying
$\|F(A)-c\|<\delta/2$ and $\|F(B)-c\|<\delta/2$. Then   $\mu_i(A)-\delta<\mu_i(B)$ for each
$i=1,\ldots,k$. Using Lemmas \ref{LemmaDenseSubset} and \ref{LemmaMain}, find a clopen subset $A'\subset A$ such that for each $i=1,\ldots,k$
$$\mu_i(A)-\delta<\mu_i(A')<\mu_i(B).$$
By Lemma \ref{LemmaMain}
 there exists an element $g\in G$ such that $g(A')\subset B$. Clearly,
  $d(g(A),B)<\delta$. Observe that $$\delta<\frac{\min\{\mu_i(A),\mu_i(B)\}}{2m}$$
 for each $i=1,\ldots,k$. Now using Proposition \ref{PropositionPropertiesProj}(2) and  repeating the arguments of Lemma \ref{LemmaContinuityTraceOnPA}
we obtain that $|tr(P^{A})-tr(P^B)| = |tr(P^{g(A)})-tr(P^B)|<\varepsilon$.
\end{proof}

 Fix an arbitrary sequence $x = (x_1,\ldots,x_k)\in (0,1]^k$. Then
there exists a sequence of clopen sets $\{A_n\}_{n\geq 1}$ such that
$x = \lim\limits_{n\to\infty} F(A_{n})$.
Lemma \ref{LemmaContinuityofPhi} applied to the sequence $\{F(X\setminus A_{n})\}_{n\geq 1}$ implies that the sequence $\{tr(P^{X\setminus
A_n})\}_{n\geq 1}$ is convergent and the limit does not depend on the choice of
 $\{A_n\}_{n\geq 1}$.
  Set $$\varphi (x) =
\lim\limits_{n\to\infty}tr(P^{X\setminus A_n}).$$

\begin{remark}  Notice that a priori it is not clear  that $\varphi(1,\ldots,1) = 1$. This will later follow from  the multiplicativity property of the function $\varphi$ provided the character $\chi$ is not regular, see Proposition \ref{PropositionFuncProperties}.
\end{remark}

 \begin{lemma}\label{LemmaGInvar} Let $\mu$ be a $G$-invariant ergodic measure.
Let $B_n$ be a  $G_n$-invariant clopen set such that $\mu(B_n)\to\mu(B)>0$
as $n\to\infty$, where $B$ is some clopen set. Then for any clopen set $A$
 we have that $\mu(A\cap B_{n})\to \mu(A)\mu(B)$ as $n\to\infty$.
 \end{lemma}
 \begin{proof} Consider a sequence of probability measures $\{\mu_n\}$ given by $$\mu_n(C):=\frac{\mu(C\cap B_n)}{\mu(B_n)}.$$ Since the set of all probability measures on $X$ is weakly compact, we can find a subsequence $\{n_l\}$ such that the sequence $\{\mu_{n_l}\}$  weakly converges to some probability measure $\nu$. It will follow from the proof below that $\nu$ is the only accumulation point of the sequence $\{\mu_n\}$. Thus, the measures $\mu_n$ converge to $\nu$. To simplify the notation, we will drop the extra index $l$.

  Observe that $\mu_n$ is $G_n$-invariant. Indeed, for $g\in G_n$ and a Borel set $C$, we have that $$\mu_n(g(C)) = \frac{\mu(g(C)\cap B_n)}{\mu(B_n)} = \frac{\mu(g(C\cap B_n))}{\mu(B_n)} = \frac{\mu(C\cap B_n)}{\mu(B_n)} = \mu_n(C).$$
 This implies that the limiting measure $\nu$ is $G$-invariant.   Note also that   $\nu(C)\leq  \mu(C)/\mu(B)$ for any closed set (and thus for any set)  $C$. Set $\delta = \mu(B)$. Hence $\mu = \delta \nu + (\mu - \delta \nu)$.  Since $\mu - \delta\nu$ is a (non-negative) $G$-invariant measure, the ergodicity of $\mu$ implies that $\mu = \nu$.

  Using the definition of $\nu$, we obtain that
 $$\mu_n(A) = \frac{\mu(B_n\cap A)}{\mu(B_n)}\to \mu(A).$$
 Thus, $\mu(B_n\cap A)\to \mu(A) \mu(B)$.
 \end{proof}

  For two sequences $x = (x_i)$ and $y = (y_i)$ from $[0,1]^k$ we will write  $x\leq y$ if $x_i\leq y_i$ for every $i$. The following proposition shows that  $\varphi$ is a continuous, monotone,  and multiplicative function. Notice that the multiplicativity of $\varphi$ can be seen as a version of the multiplicativity of  characters for $S(\infty)$.

\begin{proposition}\label{PropositionFuncProperties} (1) The function $\varphi : (0,1]^k\to [0,1]$ is continuous.

(2) The function $\varphi$ is monotone, in the sense that,  $a\leq b$, $a, b\in (0,1]^k$,  implies $\varphi(a)\leq \varphi(b)$.

(3) The function $\varphi$ is multilicative, i.e. for $a,b\in (0,1]^k$ one has that $$\varphi(a\cdot b) = \varphi(a) \varphi(b),\mbox{ where }a\cdot b = (a_1 b_1,\ldots, a_k b_k).$$
\end{proposition}
\begin{proof}(1) The continuity of $\varphi$ immediately follows  from the definition
of $\varphi$ and Lemma \ref{LemmaContinuityofPhi}.

(2) First assume that $a,b\in M$ and $a<b$. Choose clopen sets $A$ and $B$ such that $a = F(A)$ and $b=F(B)$.  By Lemma
\ref{LemmaMain} there exists an element $g\in G$ such that $g(A)\subset B$. It follows
that $$\varphi(a)=tr(P^{X\setminus A})\leq tr(P^{X\setminus B})=\varphi(b).$$ The density of $M$ in $(0,1]^k$ and continuity of $\varphi$ imply that $\varphi(a)\leq \varphi(b)$ for any $a,b\in (0,1]^k$ with $a\leq b$.

(3) Let $a,b\in (0,1]^k$. Assume first of all that $a\in M$. Let $A$ be a
 clopen set such that $a = F(A)$. Using Lemma \ref{LemmaDenseSubset}, we can find a  sequence $\{B_n\}_{n \geq 1}$ of clopen sets with
  $\lim\limits_{n\rightarrow\infty}F(B_n) = b$ and $B_n$ being $G_n$-invariant.

   Lemma \ref{LemmaGInvar} implies that, for every $i=1,\ldots,k$,
$\lim\limits_{n\rightarrow\infty}\mu_i(A\cap B_{n}) = a_ib_i$. By the continuity of $\varphi$, we have that
$$\varphi(a\cdot b)=\lim\limits_{n\rightarrow\infty}\varphi(F(A\cap B_{n})).$$
Consider a sequence of projections $\{P^{X\setminus B_{n}}\}_{n\geq 1}$.
Passing to a subsequence (we keep the same notation),
we can assume that $P^{X\setminus B_{n}}$ is weakly convergent
 to an operator $P$, see \cite[Theorem 5,1,3]{kadison_ringrose:I}.
The $G_n$-invariance of the set $X\setminus B_{n}$ implies that
the projection $P^{X\setminus B_{n}}$ commutes with all operators of the
form  $\pi(g),g\in G_{n}$, see Proposition \ref{PropositionPropertiesProj}.
 Therefore, the operator  $P$ belongs to the center of the algebra $\mathcal M_\pi$ and, hence, is a scalar operator $P = cI$. Then by the continuity of $\varphi$, we get that $$\varphi(b)  = \lim\limits_{n\to\infty}\varphi(F(B_{n})) = \lim\limits_{m\to\infty}tr(P^{X\setminus B_{n}})  = tr(P) = c.$$

Using Proposition \ref{PropositionPropertiesProj}, we obtain that
\begin{eqnarray*}\varphi(F(A\cap B_{n})) & = & tr\left(P^{X\setminus (A\cap B_{n})}\right) \\
& = &  tr\left(P^{X\setminus A} P^{X\setminus B_{n}}\right) \\
& \rightarrow  &
tr\left(P^{X\setminus A}\right)\varphi(b)\mbox{ (as }n\to\infty\mbox{)} \\
& = & \varphi(a)\varphi(b).\end{eqnarray*} Therefore, $\varphi(a\cdot b)  = \varphi(a)\varphi(b)$ for all $a\in M$ and $b\in (0,1]^k$.
The continuity of $\varphi$ implies that $\varphi(a\cdot b)=\varphi(a)\varphi(b)$ for all $a,b\in (0,1]^k$.
\end{proof}

%%
%  Theorem formula for the trace

Now we are ready to establish the main result of the paper.  Recall that the character $\chi$ is regular if $\chi(g) = 0$ for all group elements $g$ different from the identity.

\begin{theorem}\label{TheoremFormulaCharacters} Let $G$ be the full group of a Bratteli diagram. Let  $\chi: G\rightarrow \mathbb C$ be an non-zero  indecomposable character of $G$. Assume that the action of $G$ on the Bratteli diagram has only a finitely many ergodic measures, say $\mu_1,\ldots,\mu_k$, and that either the group $G$ is simple or  the character $\chi$ is continuous with respect to the uniform topology. Then there exist $\alpha_1,\ldots,\alpha_k\in \{0,1,\ldots,\infty\}$ such that
\begin{equation}\label{EquationFormulaCharacter}\chi(g) = \mu_1(Fix(g))^{\alpha_1}\mu_2(Fix(g))^{\alpha_2}\cdots\mu_k(Fix(g))^{\alpha_k}\end{equation} for any  $g\in G$. If some $\alpha_i = \infty$, then the character $\chi$ is regular.
\end{theorem}
\begin{proof} (1) For $i=1,\ldots,k$ and $x\in \mathbb R$, set $e_i(x) = (1,\ldots,1,x,1,\ldots,1)$, where $x$ stands on the $i$-th position.
Define a function $\varphi_i:(0,1]\rightarrow [0,1]$ by $\varphi(x) = \varphi(e_i(x))$.
 Proposition \ref{PropositionFuncProperties} implies that $$\varphi_i(x\cdot y)=\varphi(e_i(x\cdot y)) = \varphi(e_i(x)\cdot e_i(y)) = \varphi_i(x)\cdot \varphi_i(y)\mbox{ for all }x,y\in (0,1].$$

If $\varphi_i(x)=0$ for some $x \in (0,1]$, then the multiplicativity of $\varphi_i$ implies that $\varphi_i(y)=0$ for every $y\in (0,1]$. Since $\varphi((x_1,\ldots,x_k)) = \varphi_1(x_1)\cdots \varphi_k(x_k)$, we get that $\varphi(x) = 0$ for every $x\in (0,1]$ and  $\chi (g) =  0$, for all $g\neq 1$, by Theorem \ref{TheoremCharactersProjectors}. Thus, the character $\chi$ is regular.

If $\chi$ is not regular, then  $\varphi_i(x)\neq 0$ for all $x\in(0,1]$ and $i=1,\ldots, k$. Set $f_i(t)=-\log\varphi_i(e^{-t}).$ Then $f_i$ is continuous and additive on $\mathbb R_+ = \{t > 0 : t\in \mathbb R\}$. Thus, solving Cauchy's functional equation, we can find a real number $\alpha_i\geq 0$ such that $f_i(t) = \alpha_i t$ for all $t\in \mathbb R_+$. Therefore, $\varphi_i(x) = x^{\alpha_i}$ for every $x\in (0,1]$.
Thus, $$tr(P^A) = \varphi(F(X\setminus A)) = \mu_1(X\setminus A)^{\alpha_1}\mu_2(X\setminus A)^{\alpha_2}\cdots\mu_k(X\setminus A)^{\alpha_k}$$ for any
clopen set $A$.  Now, the formula (\ref{EquationFormulaCharacter})  follows from Theorem \ref{TheoremCharactersProjectors}.

(2) It only remains to show that the parameters $\alpha_1,\ldots,\alpha_k$ are integers.
 Fix   $1\leq i\leq k$ and  $\epsilon>0$. By Lemma \ref{LemmaDenseSubset},
there exists a clopen set $A$
such that $\mu_i(A)>1-\epsilon$ and $\mu_j(A)<\epsilon$ for $j\neq i$.

For every $n\in\mathbb{N}$, using Lemma \ref{LemmaMain}  choose mutually disjoint  $G$-equivalent clopen sets

$$A_1^{(n)},A_2^{(n)},\ldots, A_{2^n}^{(n)}\subset A\mbox{ such that }
\mu_i\left(A\setminus\bigcup_{j = 1}^{2^n} A_j^{(n)}\right)<\varepsilon.
$$ Set $$c_j^{(n,\varepsilon)} = \mu_j\left(\bigsqcup_{p=1}^{2^n} A_p^{(n)}\right).$$ Observe that
$
c_i^{(n,\varepsilon)}>1-2 \varepsilon$  and $c_j^{(n,\varepsilon)}<\varepsilon$  for all $j\neq i.$ Choose involutions
$s_1,\ldots,s_{2^n-1}$
such that $$s_j(A_j) = A_{j+1}
\text{ and }supp(s_j)=A_j\cup A_{j+1}\text{ for each }1\leq j < 2^n.$$

Denote by $H_n$ the subgroup generated by $s_1,\ldots, s_{2^n-1}$. Notice that $H_n$ is isomorphic to the symmetric group $S(2^n)$ on $2^n$ elements. Denote by $f_{n,\varepsilon}$ the isomorphism of $S(2^n)$ onto  $H_n$.

Let $s\in S(2^n)$ and $fix(s) = \{1\leq j\leq 2^n : s(j) = j\}$.  Since
$\mu_j(A_p^{(n)})=c_j^{(n,\varepsilon)}/2^n$ for each $1\leq j\leq k,1\leq p\leq 2^n$, we get that
\begin{equation}\label{EquationFormulaCharacterAux}\mu_j(Fix(f_{n,\varepsilon}(s))) = (1-c_j^{(n,\varepsilon)})+c_j^{(n,\varepsilon)}\frac{card(fix(s))}{2^n},\end{equation}  for each $1\leq j\leq 2^n$ and $s\in S(2^n)$.

Set  $\rho_{n,\epsilon}(s)=\chi_\alpha(f_{n,\varepsilon}(s))$ for every $s\in S(2^n)$. Then $\rho_{n,\epsilon}$ is a character of the group $S(2^n)$. It follows from the equations (\ref{EquationFormulaCharacter}) and (\ref{EquationFormulaCharacterAux}) that

$$\rho_n(s) \stackrel{def}{=} \lim\limits_{\varepsilon\to 0}\rho_{n,\epsilon}(s) =
 \left( \frac{card(fix(s))}{2^n}\right)^{\alpha_i},\;s\in S(2^n),$$
is a character of $S(2^n)$.
Consider the full group $S(2^\infty)$ of the Bratteli diagram
associated to the $2$-odometer, see \cite{dudko:2011} for description of this group.
Let $\nu$ be the unique invariant measure for $S(2^\infty)$.
Consider the function $\rho(s) = \nu(Fix(s))^{\alpha_i}$,
 $s\in S(2^\infty)$. Notice that the symmetric group $S(2^n)$
can be seen as a subgroup of $S(2^\infty)$ that permutes only
the first $n$ edges of every infinite path. Moreover,
 it is easy to see that $\rho(s) = \rho_n(s)$
for every $s\in S(2^n)$ and $n\geq 1$.
Since for every $n\geq 1$  the function $\rho_n: S(2^n)\rightarrow
 \mathbb C$ is a character, we get that  $\rho: S(2^\infty)\rightarrow
 \mathbb C$ is also a character.

Proposition 12 in \cite{dudko:2011} implies that the function $\rho$ can be a character only if $\alpha_i$ is an integer. We would like to mention that the proof of  \cite[Proposition 12]{dudko:2011}  is self-contained and is independent of other results from \cite{dudko:2011}.

\end{proof}

\begin{remark} (1) The theorem above implies that if $\{A_n\}$ is a sequence of clopen sets such that $\mu(A_n)\to 0$ for every  $\mu\in\mathcal M(G)$, then $tr(P^{A_n})\to 1$ as $n\to\infty$ provided the representation is built over a non-regular character.

(2) If some number $\alpha_i = 0$ in the theorem above, then the trace does not depend on the measure $\mu_i$.

\end{remark}

 Theorem \ref{TheoremFormulaCharacters} implies that non-regular indecomposable
 characters on {\it simple} full grops
are automatically continuous with respect to the uniform topology. Observe that this is not true if the group
$G$ is not simple. For instance, characters of the form $\chi(g)=\rho([g])$
(see Conjecture in Section \ref{SectionPreliminaries}) are discontinuous if $\rho$ is non-trivial.
 We would like to mention that the automatic continuity
(in a  stronger sense) holds for many ``big''
transformation groups  preserving various structures,
see, for example, \cite{kechris_rosendal:2007},
\cite{rosendal_solecki:2007}, \cite{rosendal:2009},
 and \cite{kittrell_tsankov:2010}. Using Proposition \ref{PropositionAutomaticContinuity} and Theorem \ref{TheoremFormulaCharacters}, we get the following result.

\begin{corollary}  Let $G$ be the simple full group of a Bratteli diagram $B$. Let $X$ be the path-space of the diagram $B$. Assume that the system $(X,G)$ only has a finite number of ergodic measures. Then any finite factor representation of the group $G$ corresponding to a non-regular character is automatically continuous with respect to the uniform topology.

\end{corollary}

As a corollary of Theorem \ref{TheoremFormulaCharacters}, we also obtain the classification of continuous characters for full groups of minimal dynamical systems.  Let $T: X\rightarrow X$ be a homeomorphism of a Cantor set $X$. Assume that the dynamical system $(X,T)$ is minimal, i.e every orbit of $T$ is dense in $X$. Denote by $[T]$ the set of all homeomorphisms $S$ of $X$ such that $Sx = T^{n_S(x)}x$, $x\in X$, for some function $n_S:X\rightarrow \mathbb Z$.
Then  $[T]$ is a group, termed the {\it full group of the system $(X,T)$}, see \cite{giordano_putnam_skau:1999} and \cite{bezuglyi_medynets:2008}   for more details on this group.

Assume that the system $(X,T)$  has a finitely many of ergodic measures, say
 $\mu_1,\ldots,\mu_k$. Endow the group $[T]$ with the uniform topology.
 Note that the metric generating this topology is defined by $$D(g,h)=
 \max_{i=1,\ldots,k}\mu_i(\{x\in X : g(x)\neq h(x)\}).$$
Denote by $\pi: [T]\rightarrow \mathcal M_\pi$ a continuous factor representation of the full group  $[T]$.

Each minimal dynamical system can be represented as a Vershik map acting on the path space of some Bratteli diagram, see \cite{herman_putnam_skau:1992}. Let $G$  be the full group of the Bratteli diagram associated to the system $(X,T)$.
 Observe that the group $G$ is dense in $[T]$ with respect to the uniform topology, see \cite[Theorem 2.8]{bezuglyi_dooley_kwiatkowski:2006-1} or \cite[Theorem 2.1]{medynets:2007} for a more general result.  This shows that the group $\pi(G)$ generates the von Neumann algebra $\mathcal M_\pi$.

Let $\chi$ be an indecomposable character of $[T]$ corresponding to the representation $\pi$. Since $\pi(G)$ generates $\mathcal M_\pi$, the restriction of $\chi$ onto $G$ is an indecomposable character of $G$. Using Theorem  \ref{TheoremFormulaCharacters}, we get   the following result describing the structure of continuous characters on the full group $[T]$. In particular, this result  gives an algebraic criterion of unique ergodicity for $(X,T)$.

\begin{corollary}\label{CorollaryCharactersMinimalSystems}  Let $(X,T)$ be a Cantor minimal system. Assume that $T$ has only a finitely many ergodic measures $\mu_1,\ldots,\mu_k$. Let $\chi: [T]\rightarrow \mathbb C$ be a continuous (with respect to the uniform topology) indecomposable character. Then  there exist parameters $\alpha_1,\ldots,\alpha_k\in \{0,1,\ldots, \infty \}$ such that
$$\chi(g) = \mu_1(Fix(g))^{\alpha_1}\mu_2(Fix(g))^{\alpha_2}\cdots\mu_k(Fix(g))^{\alpha_k}$$ for any $g\in [T]$.
\end{corollary}

%%%%%%%%%%%%%%%%%%%%%%%%
%

%%%%%%%%%%%%%%%%%%%%%%%

%%%%%%%%%%%%%%%%%%%%%%%%%%%%5
% EXAMPLES of PRESENTATIONS
%
\section{Examples of Factor Representations} \label{SectionExampleRepresentations}
In this section we  construct all possible (up to quasi-equivalence) finite factor representations for the simple full group $G$ under the assumption  that the group $G$ has finitely many  ergodic measures.  The constructions of group representations presented in this section have been  known in ergodic theory, see, for example, \cite{feldman_moore:2}. We  show that these representations give rise to indecomposable characters.

\subsection{Regular character}

We will first consider the case of the regular character $\chi$, i.e. $\chi(g) = 0$ if $g\neq id$ (the identity of the group) and $\chi(id) = 1$. Set $H = l^2(G)$. Define the representation $\pi$ of $G$ on $H$ by $(\pi(g)f)(h) = f(g^{-1}h)$, $h\in G$. Set $\xi = \delta_{id}$. Clearly, the vectors $\{\pi(g)\xi : g\in G\}$ are mutually orthogonal. Consider the group von Neumann algebra $\mathcal M_\pi = \{\pi(g) : g\in G\}''$. Since the group $G$ has the infinite conjugacy class property, the algebra $\mathcal M_\pi$ is a $II_1$ factor. The unique trace on $\mathcal M_\pi$ is given by $tr(Q) = (Q\xi,\xi)$, $Q\in \mathcal M_\pi$. Therefore, $\chi(g) = tr(\pi(g))$ is a regular character.

\subsection{Non-regular characters}
Denote by $\widetilde{X}$ the orbit equivalence relation of $(X,G)$, i.e.
$$\widetilde{X} = \{(x,y)\in X\times X:x\sim y\},$$
where $x\sim y$  if there is $g\in G$ such that $g(x) = y$. Note that $\widetilde{X}$ is a Borel subset of $X\times X$. Assume that  $\{\mu_1,\ldots,\mu_k\}$ is a family  of ergodic measures for the dynamical system $(X,G)$. Each ergodic measure $\mu_i$ can be lifted up to the {\it left counting measure} on $\widetilde{X}$ by $$\widetilde{\mu_i}(A) = \int_{X} card(A_x)d\mu_i(x),$$ where $A\subset \widetilde{X}$ is a Borel set and
$A_x = \{(x,y) \in A  \}$, see  \cite[Theorem 2]{feldman_moore:1} for the details.

The group $G$ has  the left and right action on the space $\widetilde{X}$
defined as $g(x,y) = (gx,y)$ and $g(x,y) = (x,gy)$,  $(x,y)\in \widetilde{X}$,
$g\in G$, respectively.  The measure $\widetilde \mu_i$ is
Borel and invariant with respect to both the left and right action of $G$.
Furthermore, the restriction of $\widetilde{\mu_i}$ on the diagonal
 $\Delta = \{(x,x) : x\in X\}$ coincides with the measure $\mu_i$ (after the identification of $\Delta$ with $X$).

Consider the Hilbert space $H_i = L^2(\widetilde{X},\widetilde{\mu_i})$.  Define the left
representation  of the group $G$ on $H$
by  $$(\pi_i(g)f)(x,y)=f(g^{-1} x,y),\;f\in H_i.$$
Similarly, the right representation $\pi'$ of $G$ is defined as  $(\pi_i'(g)f)(x,y)=f( x,g^{-1}y)$, $f\in H_i.$  The right representation will play an auxiliary and technical role in our construction.
Denote the characteristic function of the diagonal $1_\Delta(x,y)\in H_i$  by $\xi_i$.
   Set
$$\chi_i(g)=(\pi_i(g)\xi_i,\xi_i)\mbox{ for }g\in G.$$  Using straightforward computations, one can get that
\begin{equation}\label{EquationExampleRepresentation1}\chi_i(g) = \mu_i(\{x\in X : g(x) = x\})\mbox{ for all }g\in G.\end{equation}
Then for every $g,h\in G$, we have that
\begin{eqnarray*}\chi_i(h^{-1}gh) & =  & \mu_i(\{x\in X : ghx = hx\}) \\
& = &  \mu_i(h^{-1}\{y\in X : gy = y\}) \\
& = & \chi_i(g).\end{eqnarray*} Notice also  that the function $\chi_i$ is positive-definite. Hence
 $\chi_i$ is a character.

Fix non-negative integers $\alpha_1,\ldots,\alpha_k$ and set $$H=
H_1^{\otimes \alpha_1}\otimes\cdots \otimes
H_k^{\otimes \alpha_k},\;\;\xi=\xi_1^{\otimes \alpha_1}\otimes\cdots \otimes
\xi_k^{\otimes \alpha_k},\;\;\chi=\prod\chi_i^{\alpha_i}.$$

Expand the representations $\pi_i$ and $\pi_i'$ to unitary representations on $H$ by $$\pi=
\pi_1^{\otimes \alpha_1}\otimes\cdots \otimes
\pi_k^{\otimes \alpha_k}\mbox{ and }
\pi'=
(\pi'_1)^{\otimes \alpha_1}\otimes\cdots \otimes
(\pi'_k)^{\otimes \alpha_k}.$$
Thus, $\pi$ is a unitary representation of $G$ on the space $H$. Denote by $\mathcal M_\pi$ the von Neumann algebra generated by the left representation $\pi(G)$. Set $H_0=\overline{Lin(\pi(G)\xi)}.$ We can assume that the algebra $\mathcal M_\pi$ is only defined on the Hilbert space $H_0$. Hence, $\xi$ is a cyclic vector for $\mathcal M_\pi$.

We also observe that    $\xi$ is separating for the algebra $\mathcal{M_\pi}$. Assume the converse, i.e. that there is a non-zero operator $R\in \mathcal M_\pi$  such that $R\xi = 0$.
Since the representations $\pi$ and
$\pi'$ commute, we get that $0 = \pi'(g) R\xi = R\pi'(g)\xi$ for every $g\in G$. The identity
$\pi(g)\xi=\pi'(g^{-1})\xi$, $g\in G$, implies that $\{\pi'(G)\xi\}$ is dense in $H_0$. Hence $R = 0$, which is a contradiction.

Note also that the   GNS-construction applied to the character $\chi$ yields precisely the representation  $(H_0,\pi,\xi)$. It follows from the equation (\ref{EquationExampleRepresentation1}) that

\begin{equation}\label{EquationCharacterConstructionPresentations}\chi(g) = \mu_1(Fix(g))^{\alpha_1}\cdots
\mu_k(Fix(g))^{\alpha_k}\mbox{ for all }g\in G.\end{equation} Our goal is to show that the character $\chi$ is indecomposable. This will imply that $\mathcal M_\pi$ is a $II_1$-factor.  We  mention a recent work of Vershik \cite{vershik:2011}, where he  sketches a proof of the fact that the character $\chi$ constructed by one measure $\mu_i$ and $\alpha_i = 1$ is always   indecomposable provided that the group action is totally non-free.

\begin{lemma}\label{LemmaAsymptoticMultiplicativity} Let $G$ be the  full group of a Bratteli diagram.
 For any group elements $g$ and $h$ there exists a sequence $\{h_n\}_{n=1}^\infty$ such that $\chi(g h_n)\to \chi(g)\chi(h)$ for any character $\chi$  defined as in the equation (\ref{EquationCharacterConstructionPresentations}).
\end{lemma}
\begin{proof} Fix a character $\chi$ defined by measures as in
the equation (\ref{EquationCharacterConstructionPresentations}).
For each $i=1,\ldots,k$, set $c_i = \mu_i(Fix(h))$. Using Lemma
\ref{LemmaDenseSubset}, find a sequence of $G_n$-invariant clopen
 sets $B_n$ with $\mu_i(B_n)\to c_i$, $i=1,\ldots,k$, as $n\to\infty$.
 Consider  the clopen partition $$\Xi_n = \{C_v^{(i)} : v\in V_n, i=0,
\ldots p_v^{(n)}-1\}$$ determined by the level $n$ of the diagram.
Choose $h_n$  as any element of the group $G$ such that 
$Fix(h_n) = B_n$  and $h_n(C_v^{(i)}) = C_v^{(i)}$ for all $v\in V_n$
and $i=0,\ldots, p_v^{(n)}-1$.  It is not hard to see that  $supp(g h_n) = supp(g)\cup supp(h_n)$ for every $g\in G_n$. Hence $Fix(g h_n) = Fix(g)\cap Fix(h_n)$.
Using Lemma \ref{LemmaGInvar}, we obtain that
\begin{eqnarray*}\chi(gh_n) & = & \mu_1(Fix(g)\cap Fix(h_n))^{\alpha_1}\cdots
\mu_k(Fix(g)\cap Fix(h_n))^{\alpha_k} \\
&\to & \mu_1(Fix(g))^{\alpha_1}\cdots
\mu_k(Fix(g))^{\alpha_k} \cdot c_1^{\alpha_1}\cdots c_k^{\alpha_k}\\
 &= & \chi(g)\chi(h).\end{eqnarray*}
\end{proof}

\begin{theorem} Let $G$ be the  simple full group  of a Bratteli diagram $B$. Assume that the group $G$ has only finitely many ergodic measures $\{\mu_1,\ldots,\mu_k\}$ on the path-space of $B$.  Let $\chi$ be a character on $G$ of the form
$$\chi(g) = \mu_1(Fix(g))^{\alpha_1}\cdots
\mu_k(Fix(g))^{\alpha_k}\mbox{ for all }g\in G,$$
where $\alpha_1,\ldots,\alpha_k\in \mathbb{N}\cup \{0\}$.
Then $\chi$ is indecomposable.\end{theorem}
\begin{proof}  If the character $\chi$ is decomposable, then there exists a standard probability measure space
$(\Gamma,\nu)$ such that

$$\chi(g) = \int_\Gamma \chi_\gamma(g) d\nu(\gamma),$$ where $\chi_\gamma$ is an indecomposable character, see, for example, \cite[Chapter IV, Theorem 8.21]{Tak}. In view of Theorem \ref{TheoremFormulaCharacters}, there is at most a  countable number of indecomposable characters. Hence $$\chi(g) = \sum_{n\geq 1}\alpha_n\chi_n(g),$$ where $\alpha_n > 0$  and $\chi_n$ is an indecomposable character built as in the equation (\ref{EquationFormulaCharacter}).  We notice that the regular character does not show up amongst $\{\chi_n\}_{n\geq 1}$.  Observe also that $1 = \chi(id) = \sum_{n\geq 1}\alpha_n $.

Fix any two elements $a,b \in G$. Choose a sequence $\{a_n\}$ as in Lemma \ref{LemmaAsymptoticMultiplicativity}. It follows from Lemma \ref{LemmaAsymptoticMultiplicativity} and the dominated convergence theorem that $$\chi(a)\chi(b) = \lim_{n\to \infty} \chi(a b_n) = \lim_{n\to\infty}\sum_{i\geq 1}\alpha_i\chi_i(a b_n) = \sum_{i\geq 1}\alpha_i\chi_i(a)\chi_i(b).$$

Consider $g,h\in G$. Applying the Cauchy-Schwarz inequality, we get that
\begin{eqnarray*}\chi(g)\chi(h) & = & \sum_{i\geq 1}\alpha_i\chi_i(g)\chi_i(h) \\
& \leq & \left(\sum_{i\geq 1}\alpha_i\chi_i(g) \chi_i(g)\right)^{1/2} \cdot \left(\sum_{i\geq 1}\alpha_i\chi_i(h)\chi_i(h)\right)^{1/2}\\
 & = & (\chi(g)\chi(g))^{1/2}(\chi(h)\chi(h))^{1/2}\\
 & = & \chi(g)\chi(h).\end{eqnarray*}
The equality here means that $\chi_i(g) = c\chi_i(h)$ for all indices $i\geq 1$, where $c$ is some constant. Hence $\chi(g) = c \chi(h)$. Repeating these arguments for $h = id$, we get that $\chi_i(g) = q$ for all $i\geq 1 $. Thus, $\chi(g) = q = \chi_i(g)$ for all $i\geq 1$. Since the element $g$ was chosen arbitrarily, we get that $\chi \equiv \chi_{i}$ for all $i$, i.e. $\chi$ is indecomposable.
\end{proof}

It follows from Theorem \ref{TheoremFormulaCharacters} that the list of characters constructed in this section is exhaustive.

%
%
%  RATIONAL REarrangmenets

\section{Rational Permutations of the Unit Interval}\label{SectionIntervalPermutations}

In the paper \cite{goryachko_petrov:2010} (see also \cite{goryachko:2008})
Goryachko and Petrov described all indecomposable characters on the group $R$
of rational permutations of the unit interval $[0,1)$. In this section we will
apply Theorem \ref{TheoremFormulaCharacters} to get a new proof of their result.

Denote by $S(n)$ the group of permutation on
the set $\{0,1,\ldots,n-1\}$. For any $n\in \mathbb{N}$ and
a permutation $s\in S(n)$, define a bijection $g_s:[0,1)\rightarrow [0,1)$, which
  permutes subintervals
$[0,1/n),\ldots,[1-1/n,1)$ as follows $$g_s(x)=\frac{s([nx])+\{nx\}}{n},$$  where $[a]$ and $\{a\}$ is the integer and the fraction part of the number $a$, respectively.

The group $R$ is defined as the set of  all
bijections  $\{g_s : s\in S(n),\;n\in\mathbb{N}\}$. The detailed structure of the group $R$ is described in \cite{goryachko:2008}, \cite{goryachko_petrov:2010}. The representation of $R$ as an inductive limit (presented below) is taken from \cite{goryachko:2008}.

For each
$n,m\in \mathbb{N}$, define an embedding $ S(n)\rightarrow S(nm)$
 as  \begin{equation}\label{EquationParabolicEmbadding}
s\rightarrow s',\;\;s'(j)=ms([j/m])+(j\;mod\;m),\;s\in S(n),\;s'\in S(nm).\end{equation} These embeddings
are called {\it periodic}.   Observe that the map $s\rightarrow g_s$, $s\in S(n)$ and $g_s\in R$, is invariant under the periodic embeddings. We can introduce a partial ordering on $\mathbb N$ by saying that $n$ is less than $m$ if $n$ divides $m$.  The family of symmetric groups $\{S(n)\}_{n\in \mathbb N}$ is a directed family with respect to this partial ordering. The group $R$ is an inductive limit of this family with respect to the periodic embeddings, see \cite[Proposition 2]{goryachko:2008}.

 Observe that
each group $S(n)$ is embedded by a periodical embedding into a group $S(n!)$. This shows that $R$ is the inductive limit of the sequence
of groups $\{S(n!)\}_{n\geq 1}$ with respect to the periodic embeddings.

Let $B_R=(V,E)$ be a diagram such that (i) on each level $n\geq 0$, the diagram has exactly $n+1$ vertices and (ii) each vertex from $V_n$ is connected to each vertex of $V_{n+1}$, $n\geq 0$,
 exactly by one edge. In particular, $card(E_n)  = n(n+1)$.  The following figure gives  a diagrammatic representation of $B_R$.

\unitlength = 0.5cm
\begin{center}
\begin{graph}(10,7)
\graphnodesize{0.4}
%
% The top vertex
\roundnode{V0}(5, 6)
% Vertices of the first level
\roundnode{V11}(3.5,4)
\roundnode{V12}(6.5,4)
%
% Edges of the first level
\edge{V11}{V0}
\edge{V12}{V0}
%
%Vertices of the second level
\roundnode{V21}(1, 1)
\roundnode{V22}(5, 1)
\roundnode{V23}(9, 1)
%
% Edges of the second level
\edge{V21}{V11}
\edge{V21}{V12}
\edge{V22}{V11}
\edge{V22}{V12}
\edge{V23}{V11}
\edge{V23}{V12}

\end{graph}
\vskip0.4cm
$.\ .\ .\ .\ .\ .\ .\ .\ .\ .\ .\ .\ .\ .\ .\ .$
\vskip0.5cm
Fig. 1
\end{center}

\begin{proposition} The full group $G$ of the Bratteli diagram $B_R$ is isomorphic to the group of rational permutations $R$.
\end{proposition}
\begin{proof} For each $n\geq 1$, enumerate vertices from $V_n$ by the numbers $0,1\ldots,n$.
Then each finite path $p$ starting at $v_0$ and terminating at a vertex from $V_n$
can be uniquely encoded by a sequence of numbers $(a_1,\ldots,a_n)$,
where the entry $a_j\in\{0,1\ldots,j\}$
represents the vertex from $V_j$ the path $p$ goes through.

 Denote by $X_n$ the set of all
finite paths terminating at the level $n$. Let $G_n$ be the subgroup of the full group $G$ consisting of homeomorphisms that only act on $X_n$. Define a bijection $i_n: X_n\rightarrow \{0,1,\ldots,(n+1)!-1\}$ by
$$i_n(a_1,\ldots,a_n) =
a_1\frac{(n+1)!}{2!}+a_2\frac{(n+1)!}{3!}+\ldots+a_n.$$ The bijection $i_n$
induces an embedding $\varphi_n$ of $G_n$ into $S((n+1)!)$. The image
$\varphi_n(G_n)$ consists
of all permutations that do not change the remainder modulo
$n+1$. It follows that $\varphi_n(G_n)$ contains the image of $S(n!)$ in $S((n+1)!)$
under the periodic embedding. Observe that the embeddings $\varphi_n$ commute with the
embeddings $S(n!)\rightarrow S((n+1)!)$ and $G_n\rightarrow G_{n+1}$. This implies
that the inductive limit of the groups $G_n$ is isomorphic to the inductive limit
of the groups $S(n!)$.
\end{proof}

For each pair of vertices $(v,w)\in V_n\times V_m$, $n<m$, denote by $f_{v,w}^{(n,m)}$ the number of finite paths connecting $v$ to $w$. Similarly, let $f_w^{(m)}$ stand for the number of paths from the top vertex $v_0$ to the vertex $w$.  Let $\lambda$ be a $G$-invariant ergodic measure on the diagram $B_R$. By the pointwise ergodic theorem there exists an infinite path $x$ such that the measure of any cylinder set $U = U(e_1,\ldots,e_n)$ can be found by

$$\lambda(U) = \lim_{m\to\infty} f_{v,w_m}^{(n,m)} / f_{w_m}^{(m)}, $$ where $w_m$ is the vertex of level $m$ the path $x$ goes through and $v$ is the vertex of level $n$ the cylinder set $U$ terminates at,  see the details in \cite[Theorem 2]{vershik_kerov:1981} or \cite[Proposition 5.1]{bezuglyi_kwiatkowski_medynets_solomyak:2010}.
 Due to the structure of the diagram $B_R$, the quantity $f_{v,w}^{(n,m)} / f_w^{(m)}$ does not depend on the choice of vertices $v$ and $w$. Therefore, $\lambda$ is the only $G$-invariant measure. Thus, the following result is an immediate corollary of Theorem \ref{TheoremFormulaCharacters}.

\begin{corollary}[Goryachko and Petrov]\label{CorollaryCharsOnR} The functions
 $\chi_k(g)=\lambda(Fix(g))^k$,
$k\in \mathbb{N}\cup\{0,\infty\}$, $g\in R$, are the only  indecomposable characters of the group $R$ of rational permutations of the unit interval.
\end{corollary}

\begin{remark} Note that $\chi_0$ and $\chi_\infty$ in the corollary above are the identity character and  the regular character, respectively.
\end{remark}

\subsection*{Acknowledgement.}
 The second-named author would like to acknowledge the support and hospitality of Erwin Schr\"odinger Institute for Mathematical Physics in Vienna.

%%%%%%%%%%%%%%%%%%%5
%
%

%%%%%%%%%%%%%%%%%%%%%%%%%%%%%%%%%%%%%%5
%
% Bibliography e

\end{document}